\documentclass[a4paper,11pt]{article}

\usepackage[OT2,T1]{fontenc}
\DeclareSymbolFont{cyrletters}{OT2}{wncyr}{m}{n}
\DeclareMathSymbol{\Sha}{\mathalpha}{cyrletters}{"58}
\usepackage{latexsym}
\usepackage{stackrel}
\usepackage{amsfonts, amsthm}
\usepackage[all,cmtip]{xy}
\usepackage{rotating}
\usepackage{amsmath}
\usepackage{amssymb}
\usepackage{scalefnt}
\usepackage{amscd}
\usepackage[frenchb,english]{babel}
\usepackage[T1]{fontenc}
\usepackage{makeidx}
\usepackage[latin1]{inputenc}
\usepackage{fancyhdr}
\usepackage[small,nohug,heads=vee]{diagrams} 
\diagramstyle[labelstyle=\scriptstyle]
\usepackage[font={small,it},aboveskip=0pt]{caption}
\usepackage{float}
\usepackage{mathrsfs}
\usepackage[usenames]{color}
\usepackage{url}
\usepackage[shortlabels]{enumitem}
\usepackage{tabularx}
\usepackage{hyperref}

\newtheorem{teo}{Theorem}[section]
\newtheorem{prop}[teo]{Proposition} 
\newtheorem{lemma}[teo]{Lemma}
\newtheorem{cor}[teo]{Corollary}

\newtheorem*{leoconj*}{Leopoldt's Conjecture}
\newtheorem*{groconj*}{Gross's Conjecture}
\newtheorem*{greconj*}{Greenberg's Conjecture}

\theoremstyle{remark}
\newtheorem{defi}[teo]{Definition}
\newtheorem{remark}[teo]{Remark}
\newtheorem{remarks}[teo]{Remarks}
\newtheorem{example}[teo]{Example}

\newcommand{\wket}[2]{\Sha^2_{\acute{e}t}({#1},\zp{#2})}
\newcommand{\wkin}[2]{\Sha^2_{\infty}({#1},\zp{#2})}
\newcommand{\wkiw}[2]{\Sha^2_{Iw}({#1},\zp{#2})}
\newcommand{\wketp}[3]{\Sha_{\acute{e}t}^2({#1},\mathbb{Z}/p^{#2}\mathbb{Z}{#3})}

\newcommand{\hzet}[2]{H^0_{\acute{e}t}({#1},\qp/\zp{#2})}
\newcommand{\hzin}[2]{H^0_{\infty}({#1},\qp/\zp{#2})}

\newcommand{\hzetp}[3]{H^0_{\acute{e}t}({#1},\mathbb{Z}/p^{#2}\mathbb{Z}{#3})}

\newcommand{\hdet}[2]{H^2_{\acute{e}t}({#1},\zp{#2})}
\newcommand{\hdin}[2]{H^2_{\infty}({#1},\zp{#2})}
\newcommand{\hdiw}[2]{H^2_{Iw}({#1},\zp{#2})}
\newcommand{\hdetp}[3]{H^2_{\acute{e}t}({#1},\mathbb{Z}/p^{#2}\mathbb{Z}{#3})}

\newcommand{\xk}{X'_{\!k_{_{\!\scriptscriptstyle{\infty}}}}}
\newcommand{\xko}{X'^{\,\circ}_{\!k_{_{\!\scriptscriptstyle{\infty}}}}}
\newcommand{\xK}{X'_{\!K_{_{\!\scriptscriptstyle{\infty}}}}}
\newcommand{\xKo}{X'^{\,\circ}_{\!K_{_{\!\scriptscriptstyle{\infty}}}}}
\newcommand{\xKj}[1]{{X'_{\!K_{_{\!\scriptscriptstyle{\infty}}}}}^{\!\!\!\!\![{#1}]}}
\newcommand{\aKj}[1]{{A'_{K}}^{\!\![{#1}]}}
\newcommand{\aKnj}[1]{{A'_{K_n}}^{\!\!\!\![{#1}]}}
\newcommand{\akn}{A'^{\,\Gamma_n}_{k_\infty}}

\newcommand{\zp}{\mathbb{Z}_p}
\newcommand{\qp}{\mathbb{Q}_p}

\begin{document}
\title{\textsc{On the splitting of the exact sequence\\ relating the wild and tame kernels}}
\author{Luca Caputo and Abbas Movahhedi}
\date{\today}
\maketitle

\begin{abstract}
\noindent
Let $k$ be a number field. For an odd prime $p$ and an integer $i\geq 2$, let $\wket{k}{(i)}$ denote the \'etale wild kernel of $k$ (corresponding to $p$ and $i$). Then $\wket{k}{(i)}$ is contained in the finite group $\hdet{o'_k}{(i)}$, where $o'_k$ is the ring of $p$-integers of $k$. We give conditions for the inclusion $\wket{k}{(i)}\subseteq \hdet{o'_k}{(i)}$ to split. We analyze this problem using Iwasawa theory. In particular we relate the splitting of the above inclusion to the triviality of two invariants, namely the asymptotic kernels of the Galois descent and codescent for class groups along the cyclotomic tower of $k$. We illustrate our results in both split and non-split cases for quadratic number fields.
\end{abstract}

\section{Introduction and notations}

For a number field $k$, the classical wild kernel $W\!K_{2}(k)$ is the kernel of all local power norm residue symbols. It fits in the Moore exact sequence
\begin{equation}\label{sequence0}
0\to W\!K_{2}(k)\to K_2(k)\to \underset{v}{\oplus} \mu(k_v)\to \mu(k)\to 0,
\end{equation}
where $v$ runs through all non complex places of $k$, $k_v$ is the completion of $k$ at $v$ and $\mu(k_v)$ (resp. $\mu(k)$) is the group of roots of unity of $k_v$ (resp. $k$). It is known (see \cite{Ta1} and also \cite{Hu}) that, for an odd prime $p$, the $p$-primary part $W\!K_{2}(k)\{p\}$ of the wild kernel consists of the elements of infinite height of $K_2(k)\{p\}$. In particular, the inclusion $W\!K_{2}(k)\{p\}\subset K_2(k)\{p\}$ never splits, unless $W\!K_{2}(k)\{p\}=0$.

The wild kernel is actually contained in the tame kernel of $k$, denoted $K_2(o_k)$, and there is an exact sequence
$$0\to W\!K_{2}(k)\{p\}\to K_2(o_k)\{p\}\to \underset{v\mid p}{\oplus} \mu(k_v)\{p\}\to \mu(k)\{p\}\to 0.$$
In the study of the structure of $K_2(o_k)$, it is interesting to know when the inclusion $W\!K_{2}(k)\{p\}\subseteq K_2(o_k)\{p\}$ splits. Strangely enough, it seems that this question has never been studied so far. The aim of this paper is to give conditions for the inclusion $W\!K_{2}(k)\{p\}\subseteq K_2(o_k)\{p\}$ to split. Together with some additional results (e.g. the Birch-Tate formula), this sometimes allows one to completely determine the structure of $K_2(o_k)\{p\}$ (see Section \ref{examples}).

Actually, using the results of Tate (\cite{Ta2}), $W\!K_{2}(k)\{p\}$ (resp. $K_2(o_k)\{p\}$) can be interpreted as a Tate-Shafarevic (resp. cohomology) group. This leads directly to a generalization of the inclusion $W\!K_{2}(k)\{p\}\subseteq K_2(o_K)\{p\}$, as we shall now explain. 

Let $p$ be an odd prime and let $i\in\mathbb{Z}$ be an integer. Let $o'_k$ be the ring of $p$-integers of $k$, which embeds into $k_v$ for every $v\mid p$. Accordingly we have a homomorphism in cohomology
$$\lambda_i:\hdet{o'_k}{(i)}\to \underset{v\mid p}{\oplus} \hdet{k_v}{(i)}.$$ 
Here $\hdet{o'_k}{(i)}=\lim\limits_{\longleftarrow}H^2_{\acute{e}t}(o'_k,\mu_{p^n}^{\otimes i})$, where the limits are taken over the natural maps $\mu_{p^{n+m}}^{\otimes i}\to \mu_{p^n}^{\otimes i}$ of $Spec(o'_k)$-\'etale sheaves (and similarly for $\hdet{k_v}{(i)}$). 
One can define the $i$-th \'etale wild kernel of $K$ as
$$\wket{k}{(i)}=\mathrm{ker}(\lambda_i)$$
(see \cite{Sc}, \cite{N4}, \cite{Kol1}, \cite{Ba}). 
%
Note that $\wket{k}{(i)}$ (resp. $\hdet{o'_k}{(i)}$) is finite if $i\geq 1$ (resp. $i\geq 2$), see \cite{Sc}, \cite{So}. In fact $\wket{k}{(i)}$ is also conjectured to be finite when $i\leq 0$, but for the rest of this introduction we will suppose $i\geq 1$. Part of the (dual) exact sequence of Poitou-Tate (see \cite[\S 4, Satz 8]{Sc}) then reads 
\begin{equation}\label{sequence}
0\to \wket{k}{(i)}\stackrel{\iota_{k,i}}{\rightarrow} \hdet{o'_k}{(i)}\stackrel{\lambda_i}{\longrightarrow} \underset{v\mid p}{\oplus}\hdet{k_v}{(i)} \stackrel{res^*}{\longrightarrow} \hzet{k}{(1-i)}^*\to 0,
\end{equation}
where $(\--)^*$ denotes the Pontryagin dual and, by local duality,
$$\hdet{k_v}{(i)}\cong \hzet{k_v}{(1-i)}^*.$$ 
For $i=2$, there is a commutative diagram
$$
\begin{CD}
W\!K_{2}(k)\{p\}@>>> K_2(o_k)\{p\}\\
@V\wr VV@V\wr VV\\
\wket{k}{(2)}@>\iota_{k,2}>> \hdet{o'_k}{(2)},\\
\end{CD}
$$
where the left isomorphism is due to Schneider (\cite{Sc}) and the right one is due to Soul\'e (\cite{So}). Therefore our splitting problem is equivalent to that of $\iota_{k,2}$. As it is well-known, there are analogues of the above diagrams for any $i\geq 2$: the vertical arrows of these ``higher'' diagrams are still isomorphisms in view of the Quillen-Lichtenbaum conjecture, which is a consequence of the Bloch-Kato conjecture (see for example \cite[Theorem 2.7]{Kol2}) proven by Rost-Voevodsky (see \cite{We} for more details). In fact we will study the splitting of the inclusion $\iota_{k,i}$ for general $i\geq 2$ (it is easy to see that $\iota_{k,1}$ always splits, see Section \ref{code}).  
To this end, we look at this problem along the cyclotomic tower of the base field and apply Iwasawa theoretical tools to derive our results.

Let $\wkin{k}{(i)}$ (resp. $\hdin{o'_k}{(i)}$) be the direct limit of $\wket{k_n}{(i)}$ (resp. $\hdet{o'_{k_n}}{(i)}$) with respect to the restriction homomorphisms along the cyclotomic $\zp$-extension $k_\infty$ of $k$, whose $n$-th layer is denoted $k_n$. Set $K=k(\mu_p)$ and $\Delta=\mathrm{Gal}(K/k)$. Then the map $\iota_{K,\infty}=\iota_{K,\infty,i}=\lim\limits_{\longrightarrow}\iota_{K_n,i}$ fits in a commutative diagram (see Proposition \ref{nqd})
$$
\begin{CD}
\wkin{K}{(i)}@>\iota_{K,\infty}>>\hdin{o'_K}{(i)}\\
@A\wr AA@A\wr AA\\
t_{\Lambda}(BP_{K_\infty})(-i)^*@>>>t_{\Lambda}(\mathfrak{X}_{K_\infty})(-i)^*
\end{CD}
$$
assuming Leopoldt's conjecture. Here $\mathfrak{X}_{K_\infty}$ (resp. $BP_{K_\infty}$) is the Galois group of the maximal abelian pro-$p$-extension of $K_\infty$ which is unramified outside $p$ (resp. the Galois group of the union of the fields of Bertrandias-Payan over $K_n$ for all $n$, see Definition \ref{BP}). Both $\mathfrak{X}_{K_\infty}$ and $BP_{K_\infty}$ are finitely generated $\Lambda$-modules (where $\Lambda=\zp[\![\mathrm{Gal}(K_\infty/K)]\!]$ is the Iwasawa algebra) and $t_{\Lambda}(\cdot)$  denotes the $\Lambda$-torsion submodule. 

The advantage of considering limits lies mainly in the fact that, using the results of \cite{LMN}, one can easily spot a sufficient condition for the surjection
\begin{equation}\label{surje}
t_{\Lambda}(\mathfrak{X}_{K_\infty})\to t_{\Lambda}(BP_{K_\infty})
\end{equation} 
to split (as a morphism of $\Lambda$-modules), namely the triviality of the invariant $\Psi(K_\infty)$. The latter is a finite group (under Gross's conjecture for every layer $K_n$) defined as the direct limit of the groups
$$\Psi(K_n):=\ker \left(\left(X'_{K_\infty}\right)_{\mathrm{Gal}(K_\infty/K_n)}\to A'_{K_n}\right).$$
Here $X'_{K_\infty}=\varprojlim A'_{K_n}$, where $A'_{K_n}$ is the $p$-Sylow subgroup of the class group of $o'_{K_n}$ and the limit is taken with respect to norms. 

Once a condition for the splitting of the inclusion $\iota_{K,\infty}$ has been found, one may transfer it to finite layers. The most direct way is to perform Galois descent from $K_\infty$ to $K$, namely consider the commutative diagram
\begin{equation}\label{desquare}
\begin{CD}
\wket{k}{(i)} @>\iota_k>>\hdet{o'_k}{(i)}\\
@VVV@VVV\\
\qquad\wkin{K}{(i)}^{\mathrm{Gal}(K_\infty/k)}@>\iota_{K,\infty}>>\qquad\hdin{o'_K}{(i)}^{\mathrm{Gal}(K_\infty/k)}.
\end{CD}
\end{equation}
The vertical arrows of the above diagram are surjective and have the same kernel $\xKo(i-1)_{\mathrm{Gal}(K_\infty/k)}$ (see Theorem \ref{descent}). Here $\xKo$ denotes the maximal finite $\Lambda$-submodule of $\xK$. Therefore $\iota_k$ splits as soon as $\Psi(K_\infty)$ and $\xKo$ vanish. It turns out that something stronger is also true (see Proposition \ref{main} and Theorem \ref{randonnee}). For instance $\iota_k$ splits, provided that $\Psi(K_\infty)(i-1)^{\Delta}=0$ and the right vertical arrow of (\ref{desquare}) splits (as abelian groups). The latter condition holds, for example, for $k_n$, if $n$ is sufficiently large and the Iwasawa $\mu$-invariant of $X'_{K_\infty}$ is trivial (see \cite{Va2}). 

We have also developed another approach which uses codescent instead of descent. Let $\wkiw{K}{(i)}$ (resp. $\hdiw{o'_K}{(i)}$) be the inverse limit of $\wket{K_n}{(i)}$ (resp. $\hdet{o'_{K_n}}{(i)}$) over the corestriction homomorphisms. Then we have an injection $\iota_{K,Iw}=\iota_{K,Iw,i}=\lim\limits_{\longleftarrow}\iota_{K_n,i}$ and a commutative diagram 
\begin{equation}\label{cosquare}
\begin{CD}
\wket{k}{(i)} @>\iota_k>>\hdet{o'_k}{(i)}\\
@A\wr AA@A\wr AA\\
\qquad\wkiw{K}{(i)}_{\mathrm{Gal}(K_\infty/k)}@>\iota_{K,Iw}>>\qquad\hdiw{o'_K}{(i)}_{\mathrm{Gal}(K_\infty/k)}
\end{CD}
\end{equation}
whose vertical arrows are isomorphisms. Hence, for $\iota_k$ to split it is sufficient that $\iota_{K,Iw}$ splits in the category of $\Lambda[\Delta]$-modules and we shall give a condition for that in Theorem \ref{main2} (see also Lemma \ref{pavia}). When the $p$-adic primes are totally ramified in $K_\infty/K$, the above condition simply reduces to $\Psi(K)(i-1)^\Delta=0$, see Corollary \ref{simplecase}. 


One can also follow a quite different approach, based on the following remark. The inclusion $\iota_k:\wket{k}{(i)}\to\hdet{o'_k}{(i)}$ splits if and only if, for every $n\in\mathbb{N}$, the induced homomorphism $\wket{k}{(i)}/p^n\to \hdet{o'_k}{(i)}/p^n$ stays injective. In Section \ref{crit}, we express the above criterion in terms of class groups, following the strategy of \cite{Ca} and using in particular Schneider's isomorphism (\cite{Sc})
$$\xK(i-1)_{\mathrm{Gal}(K_\infty/k)}\cong \wket{k}{(i)} \quad (i\ne 1).$$
This turns out to be particularly useful to give examples of number fields $k$ for which $\iota_k$ does not split. We also use our criterion to give a different proof of Corollary \ref{simplecase}.

We conclude by applying our results to the inclusion 
$$W\!K_{2}(k)\{3\}\subseteq K_2(o_k)\{3\},$$
for families of quadratic fields $k$, giving examples for both split and non-split cases. In particular we are able to confirm some of the predictions of \cite{BG} on the exact structure of wild and tame kernels of imaginary quadratic fields.

\subsection*{Notations}
Let $k$ be a number field and let $p$ be an odd prime, we now give the most relevant notations used in the paper:

\noindent
\begin{tabular}{lcl}
$\mu_{p^n}$, $\mu_{p^\infty}$&&group of $p^n$-th (resp. $p$-power order) roots of unity in an algebraic closure $\overline{k}$ of $k$;\\
$K$&& $k(\mu_p)$;\\
$k_\infty$&&cyclotomic $\zp$-extension of $k$ inside $\overline{k}$;\\
$k_n$&&$n$-th layer of $k_\infty$;\\
$\Gamma_n$&&Galois group of $k_\infty/k_n$, which is identified with $\mathrm{Gal}(K_\infty/K_n)$ ($\Gamma=\Gamma_0$);\\
$\mathcal{G}_n$&&Galois group of $K_\infty/k_n$ ($\mathcal{G}=\mathcal{G}_0$);\\
$\Delta$&&Galois group of $K_\infty/k_\infty$, which is identified with $\mathrm{Gal}(K/k)$;\\
$A'_{k_n}$&&$p$-Sylow subgroup of the ideal class group of the ring of $(p)$-integers of $k_n$;\\
$A'_{k_\infty}$ &&direct limit of $A'_{k_n}$ with respect to the maps induced by the inclusions $k_n\subset k_{n+1}$;\\
$X'_{k_\infty}$&&projective limit of $A'_{k_n}$ with respect to the maps induced by the norms $k_{n+1}\to k_n$;\\
$\mathfrak{X}_{k_n}$, $\mathfrak{X}_{k_\infty}$&&Galois group of the maximal pro-$p$-abelian extension of $k_n$ (resp. $k_\infty$) unramified outside $p$.\\
\end{tabular}
\linebreak
\\

Let $M$ be a (topological) $\zp[\mathrm{Gal}(\overline{k}/k)]$-module. The Pontryagin dual $M^*=\mathrm{Hom}_{\zp}(M,\mathbb{Q}_p/\mathbb{Z}_p)$ of $M$ is a $\zp[\mathrm{Gal}(\overline{k}/k)]$-module with the action defined by $(g\phi)(m)=\phi(g^{-1}m)$, for any $g\in \mathrm{Gal}(\overline{k}/k)$, $m\in M$ and $\phi\in M^*$. We also consider, for any $i\in \mathbb{Z}$, the $i$-th Tate twist of $M$, denoted $M(i)=M\otimes_{\zp}\zp(i)$, as a $\zp[\mathrm{Gal}(\overline{k}/k)]$ with the action $g(m\otimes x)=g(m)\otimes g(x)$. 
 
 
If $A$ is a ring and $M$ is an $A$-module, we will denote by $t_A(M)$ the torsion submodule of $M$. For $a\in A$, $M\{a\}$ denotes the submodule of elements of $M$ which are killed by a power of $a$. If $G$ is a group acting on $M$, we denote by $M^G$ (resp. $M_G$) the submodule of invariants (resp. the quotient of coinvariants) of $M$, \textit{i.e.} the maximal submodule (resp. the maximal quotient) of $M$ on which $G$ acts trivially. 

We also recall three classical conjectures we shall often refer to throughout the paper.

\begin{leoconj*}
The $\zp$-module $\mathfrak{X}_{k}$ has rank $r_2(k)+1$, where $r_2(k)$ is the number of complex places of $k$.
\end{leoconj*}

\begin{groconj*}
The group $(\xk)_\Gamma$ (or equivalently $(\xk)^\Gamma$) is finite.
\end{groconj*}

\begin{greconj*}
If $k$ is totally real, then $\xk$ is finite.
\end{greconj*}

If $E/k$ is a finite extension and Leopoldt's conjecture (resp. Gross's conjecture) holds for $E$, then it holds for $k$ too. Moreover, both Leopoldt's and Gross's conjectures are known to hold if $k/\mathbb{Q}$ is abelian. 

\section{Background}
This section is devoted to the description of some Iwasawa theoretical objects which will allow us to formulate our main results.

Let $p$ be an odd prime. For a number field $k$, set $K=k(\mu_p)$, where $\mu_p$ is the group of $p$-th roots of unity in an algebraic closure $\overline{k}$ of $k$, and $\Delta=\Delta_k=\mathrm{Gal}(K/k)$.  Let $k_\infty$ denote the cyclotomic $\zp$-extension of $k$ and set $\Gamma=\Gamma_k=\mathrm{Gal}(k_\infty/k)$. As usual, we shall denote by $\Lambda=\Lambda_k=\zp[\![\Gamma_k]\!]$ the completed group algebra of $\Gamma_k$. For every $n\in\mathbb{N}$, we set $\Gamma_n=\Gamma^{p^n}$ and denote by $k_n$ the subfield of $k_\infty$ fixed by $\Gamma_n$, so that $k_\infty=\cup_{n\in\mathbb{N}} k_n$. 

Let $o'_k$ be the ring of $p$-integers of $k$ and, for any $i\in\mathbb{Z}$, let $\wket{k}{(i)}$ and $\hdet{o'_k}{(i)}$ denote the higher wild and tame kernels of $k$, respectively, as defined in the introduction. We will use the following notation:
$$\wkiw{k}{(i)}=\lim_{\longleftarrow}\wket{k_n}{(i)} \quad\textrm{and}\quad\hdiw{o'_k}{(i)}=\lim_{\longleftarrow}\hdet{o'_{k_n}}{(i)},$$
where the limits are taken over corestriction maps, and
$$\wkin{k}{(i)}=\lim_{\longrightarrow}\wket{k_n}{(i)} \quad\textrm{and}\quad\hdin{o'_k}{(i)}=\lim_{\longrightarrow}\hdet{o'_{k_n}}{(i)},$$
where  the limits are taken over restriction maps. Note that the inclusions $\iota_{k_n}:\wket{k_n}{(i)}\to \hdet{o'_{k_n}}{(i)}$ induce inclusions 
$$\iota_{k,Iw,i}=\iota_{k,Iw}:\wkiw{k}{(i)}\to \hdiw{o'_k}{(i)}\quad\textrm{and}\quad\iota_{k,\infty,i}=\iota_{k,\infty}:\wkin{k}{(i)}\to \hdin{o'_k}{(i)}.$$

It is well known that both higher tame and wild kernels behave well under Galois co-descent along the cyclotomic $\zp$-extension. More precisely,  for any $n\in \mathbb{N}$ and $i\geq 2$, the natural projection homomorphisms induce isomorphisms 
\begin{equation}\label{codescentwh}
\wkiw{k}{(i)}_{\Gamma_n}\cong \wket{k_n}{(i)} \quad\textrm{and}\quad \hdiw{o'_k}{(i)}_{\Gamma_n}\cong \hdet{o'_{k_n}}{(i)},
\end{equation} 
 (see \cite[Corollary 2.7]{KM}).

On the other hand the natural descent homomorphisms 
$$\wket{k_n}{(i)}\to\wkin{k}{(i)}^{\Gamma_n}\quad\textrm{and}\quad \hdin{k_n}{(i)}\to \hdin{o'_k}{(i)}^{\Gamma_n}$$
are not in general injective but still surjective, as we shall recall in Theorem \ref{descent} below.

First we need some notation.  For any number field $k$, let $A'_k$ denote the $p$-Sylow subgroup of the ideal class group of $o'_{k}$. For any $n\in\mathbb{N}$, set $\mathcal{G}_{k_n}=\mathcal{G}_{n}=\mathrm{Gal}(K_\infty/k_n)$ and we shall also write $\mathcal{G}=\mathcal{G}_k$ for $\mathcal{G}_0$. Let $X'_{k_\infty}$ denote the Galois group of the maximal abelian pro-$p$-extension of $k_\infty$ which is completely decomposed everywhere. Then, by class field theory, $X'_{k_\infty}$ is isomorphic to the inverse limit $\lim\limits_{\longleftarrow}A'_{k_n}$  with respect to the norm homomorphisms. Finally, for a finitely generated $\Lambda$-module $M$, we denote by $M^\circ$ the maximal finite $\Lambda$-submodule of $M$. Then the following theorem (more or less known) gives the description of Galois descent for higher wild and tame kernels.

\begin{teo}\label{descent}
Let $i\geq 2$ be an integer. There exists a commutative diagram of $\mathrm{Gal}(k_n/k)$-modules 
\begin{equation}\label{ladiscesa}
\begin{CD}
0@>>>\xKo(i-1)_{\mathcal{G}_n}@>>>\wket{k_n}{(i)}@>>>\wkin{k}{(i)}^{\Gamma_n}@>>>0\\
@.@|@VVV@VVV\\
0@>>>\xKo(i-1)_{\mathcal{G}_n}@>>>\hdet{o'_{k_n}}{(i)}@>>>\hdin{o'_k}{(i)}^{\Gamma_n}@>>>0,
\end{CD}
\end{equation} 
whose rows are exact and vertical arrows are injective. Moreover, if the Iwasawa $\mu$-invariant of the $\Lambda$-module $\xK$ is trivial, then both the top and the bottom exact sequences split (as abelian groups) for $n$ large. 
\end{teo}
\begin{proof}
The bottom sequence is defined and proved to be exact in \cite[Theorem 7]{Co} for $i=2$ in terms of $K$-theory and in \cite[Section 4]{N3} for general $i$ (see also \cite[Proposition 3.2, Corollary 3.3, Theorem 3.6]{KM}). 
The surjectivity of the map $\wket{k_n}{(i)}\to\wkin{k}{(i)}^{\Gamma_n}$ is noticed in \cite[page 854]{LMN} (it is an easy consequence of \cite[Lemma 1.1]{LMN} together with Schneider's isomorphism $X'_{K_\infty}(i-1)_{\mathcal{G}_n}\cong \wket{k_n}{(i)}$, see \cite[\S6, Lemma 1]{Sc}) and its kernel is indeed isomorphic to $\xKo(i-1)_{\Delta}$ as shown in \cite[Proposition 3.8]{KM}. Using the injectivity of $\wket{k_n}{(i)}\to\hdet{o'_{k_n}}{(i)}$ and the finiteness of $\xKo(i-1)_{\mathcal{G}_n}$, we deduce that the left-hand square of the diagram of the theorem is indeed commutative. The last assertion of the theorem is due to Validire (see \cite[Proposition 3.3 and Theorem 4.1]{Va2}).
\end{proof}

\begin{remark}
One can actually show that the last assertion of the above theorem holds under the weaker assumption that the Iwasawa $\mu$-invariant of the $\Lambda$-module $\xK(i-1)_{\Delta}$ is trivial (see \cite[Th\'eor\`eme 3.1.8]{Va1}).
\end{remark}

The module $\xKo$ has a relevant role in our approach. Of course $\xKo$ has other arithmetic interpretations, discovered mainly by Iwasawa, which we shall now briefly recall.

When $i=1$, $\wket{k}{(1)}$ is isomorphic to $A'_k$ (see \cite[\S 4, Satz 8]{Sc}). A classical result of Iwasawa (whose proof is essentially contained in \cite{Iw}) states that there is an exact sequence  
\begin{equation}\label{seclass}
0\to (\xko)_{\Gamma_n}\to A'_{k_n}\to \akn,
\end{equation}
where $A'_{k_\infty}=\lim\limits_{\longrightarrow}A'_{k_n}$ denotes the direct limit of class groups with respect to the maps induced by extension of ideals. The above sequence may therefore be considered as the classical analogue of the top sequence of the diagram of Theorem \ref{descent}. The splitting (as abelian groups) of the inclusion $(\xko)_{\Gamma_n}\to A'_{k_n}$ in (\ref{seclass}) for $n$ large was proved by Grandet and Jaulent (see \cite[Th\'eor\`eme]{GJ}).

Another interpretation of $\xko$ can be given in the following context. Let $\mathfrak{X}_{k}$ (resp. $\mathfrak{X}_{k_\infty}$) be the Galois group of the maximal pro-$p$-abelian extension of $k$ (resp. $k_\infty$) unramified outside $p$ and the archimedean primes. In fact we have  $\mathfrak{X}_{k_\infty}=\lim\limits_{\longleftarrow}\mathfrak{X}_{k_n}$, the limit being taken over the restriction maps. Moreover $\mathfrak{X}_{k}$ (resp. $\mathfrak{X}_{k_\infty}$) can be shown to a be finitely generated $\zp$-module (resp. $\Lambda$-module). More precisely, a well-known theorem of Iwasawa asserts that there is a pseudo-isomorphism $f:\mathfrak{X}_{k_\infty}/t_{\Lambda}(\mathfrak{X}_{k_\infty})\to \Lambda^{r_2(k)}$ where $r_2(k)$ is the number of complex places of $k$ (see \cite[Theorem 13.31]{Wa}) and $H_k=\mathrm{coker}(f)$ is independent of $f$ up to isomorphism (see \cite[\S 3]{Gr}, \cite{Ja}). Then, as explained by Greenberg in \cite[\S 5]{Gr}, one can deduce from \cite{Iw} an isomorphism of $\Lambda[\Delta]$-modules  
$$\xKo\cong H_{K}^*(1),$$
where, for a $\Lambda[\Delta]$-module $M$, we denote by $M^*=\mathrm{Hom}_{\zp}(M,\qp/\zp)$ its Pontryagin dual. Note that, if $k$ is totally real, then the plus part of $H_{K}$ is clearly trivial. In particular the above isomorphism can be used, for instance, to prove the following well-known result. 

\begin{prop}\label{coates}
If $k$ is totally real and $i\in\mathbb{Z}$ is even, then 
$$\xKo(i-1)_{\mathcal{G}_n}=0\quad\textrm{for every $n\in\mathbb{N}$.}$$ 
\end{prop}

We now come to another object which is fundamental for our approach. 
\begin{defi}
For a number field $k$, let $\Psi(k)$ denote the kernel of the natural map $(X'_{k_\infty})_{\Gamma}\to A'_{k}$.
\end{defi}
Note that, if Gross's conjecture holds for $k$, $\Psi(k)$ is a finite group. If $m\geq n$ it is easy to see that the map $(X'_{k_\infty})_{\Gamma_n}\to (X'_{k_\infty})_{\Gamma_m}\,,$ given by multiplication by $\omega_m/\omega_n$ induces a map $\Psi(k_n)\to \Psi(k_m)$. 
\begin{defi}
For a number field $k$, set
$$\Psi(k_\infty):=\lim_{\longrightarrow} \Psi(k_n),$$
where the limits are taken over the maps described above.
\end{defi}

Let $n_0=n_0(k)$ denote the smallest integer for which all $p$-adic primes are totally ramified in $k_\infty/k_{n_0}$ (in particular the map $X'_{k_\infty}\to A'_{k_n}$ is surjective for $n\geq n_0$).

\begin{lemma}\label{CPsi}
For any $j\in\mathbb{Z}$, if $\Psi(K_n)(j)_\Delta=0$ for some $n\geq n_0(K)$, then $\Psi(K_m)(j)_\Delta=0$ for every $m\geq n$ and therefore $\Psi(K_\infty)(j)_\Delta=0$.
Assume moreover that the Gross conjecture holds for all the layers $k_n$. Then for all $n$ large enough the groups $\Psi(k_n)$ stabilize and the natural maps $\Psi(k_n)\to \Psi(k_\infty)$ become isomorphisms (in particular $\Psi(k_\infty)$ is finite).
\end{lemma}
\begin{proof}
The proof of the first assertion is an easy generalization of \cite[proof of Corollary 1.6]{LMN}. The last assertion is \cite[Lemma 1.3]{LMN}.
\end{proof}

\begin{remark}\label{oujda}
If there is only one prime above $p$ in $k_\infty$ then $\Psi(k_{n})=0$ for every $n\geq n_0(k)$ (this can be easily deduced for example from \cite[Lemma 13.15]{Wa}) and therefore $\Psi(k_\infty)=0$. Note also that if $k$ is a totally real field satisfying Greenberg's conjecture, then one can easily show that the norm induces isomorphisms $A'_{k_m}\cong A'_{k_n}$ for $m\geq n$ large enough, implying that $\Psi(k_\infty)$ is trivial.
\end{remark}

The invariant $\Psi(k_\infty)$ has at least two interesting interpretations, described in \cite{LMN}. To explain the first one, set, for any $n\in \mathbb{N}$,
$$C_{k_n}:=\mathrm{coker}(A'_{k_n}\to \akn).$$ 
Then, if Gross's conjecture holds for $k_n$, $C_{k_n}$ is a finite group (actually $(\xk)_{\Gamma_n}$ is finite if and only if $\akn$ is, as is easily seen using \cite[Theorem 11]{Iw}). As is shown in \cite[Theorem 1.4]{LMN}, if $n$ is large enough, $C_{k_n}$ is independent of $n$ (up to isomorphism) and indeed isomorphic to $\Psi(k_\infty)$.

The second interpretation of $\Psi(k_\infty)$, which will be the most relevant for us, is in terms of Galois theory.  

\begin{defi}\label{BP}
Let $K$ be a number field containing $\mu_p$. The field of Bertrandias-Payan $K^{BP}$ over $K$ is the compositum of all the (cyclic) $p$-extensions of $K$ which are embeddable in cyclic extensions of degree $p^m$, for all $m\geq 1$. Note that $K^{BP}/K$ is a Galois extension and we set $BP_K=\mathrm{Gal}(K^{BP}/K)$. 
\end{defi}

In fact $K^{BP}$ coincides with the compositum of all extensions of $K$ which are locally $\zp$-embeddable for any finite place (see \cite[Lemme 4.1]{N1}). Using this observation, it is not difficult to show that, if $M_K$ denotes the maximal abelian pro-$p$-extension of $K$ unramified outside $p$ (thus $\mathfrak{X}_K=\mathrm{Gal}(M_K/K)$), then $K^{BP}\subseteq M_K$ and accordingly the restriction induces a surjection
\begin{equation}\label{XBP}
\mathfrak{X}_{K}\twoheadrightarrow BP_K.
\end{equation} 
Moreover, if $E$ is a number field containing $K$, then one easily sees that $K^{BP}\subseteq E^{BP}$ and therefore the restriction defines a map $BP_E\to BP_K$. In particular we can consider $BP_{K_\infty}=\lim\limits_{\longleftarrow}BP_{K_n}$, the limit being taken over the restriction maps. Set also $K^{BP}_\infty=\cup_{n\in \mathbb{N}}K^{BP}_n$, so that $\mathrm{Gal}(K^{BP}_\infty/K_\infty)=BP_{K_\infty}$ (note, however, that in general $K^{BP}_\infty$ is strictly contained in the compositum of all the $p$-extensions of $K_\infty$ which are embeddable in cyclic extensions of degree $p^m$ for all $m\geq 1$). 
Now $K_\infty^{BP}$ is contained in $M_{K_\infty}=\cup M_{K_n}$ (thus $M_{K_\infty}$ is the maximal abelian pro-$p$-extension of $K_\infty$ which is unramified outside $p$ and $\mathfrak{X}_{K_\infty}=\mathrm{Gal}(M_{K_\infty}/K_\infty)$) and the inverse limits of the surjections (\ref{XBP}) gives a surjection of $\Lambda$-modules
$$\mathfrak{X}_{K_\infty}\twoheadrightarrow BP_{K_\infty}.$$
Using the arguments of \cite[proof of Theorem 4.2]{N1}, one can show that the kernel of the above map is a $\Lambda$-torsion module, giving a surjection between the torsion submodules
\begin{equation}\label{slt} 
t_\Lambda(\mathfrak{X}_{K_\infty})\twoheadrightarrow t_\Lambda(BP_{K_\infty}).
\end{equation} 

Let $T_{K_\infty}$  be the subextension of $M_{K_\infty}$ which is fixed by the torsion submodule $t_\Lambda(\mathfrak{X}_{K_\infty})$ of $\mathfrak{X}_{K_\infty}$ and $N'_{K_\infty}$ be the (Galois) extension of $K_\infty$ which is obtained by adjoining to $K_\infty$ the $p^m$-th roots of all the $p$-units of $K_\infty$, for all $m\in\mathbb{N}$. We have $T_{K_{\infty}}\subseteq N'_{K_{\infty}}$, thanks to \cite[Theorem 15]{Iw}, and $T_{K_\infty}\subseteq K_\infty^{BP}$, by (\ref{slt}). If we set $N''_{K_\infty}=N'_{K_\infty}\cap K^{BP}_\infty$, then the promised Galois theoretical interpretation of $\Psi(K_\infty)$ is given by the following result.

\begin{teo}
Assume that the Gross conjecture holds for all layers $K_n$. Then there is an isomorphism of $\Lambda[\Delta]$-modules
$$\Psi(K_\infty)^*(1)\cong\mathrm{Gal}(N''_{K_\infty}/T_{K_\infty}).$$
\end{teo}
\begin{proof}
See \cite[Th\'eor\`eme 2.5]{LMN}.
\end{proof}

The above setting can therefore be summarized by the following diagram of fields which are all Galois extensions of $k$:

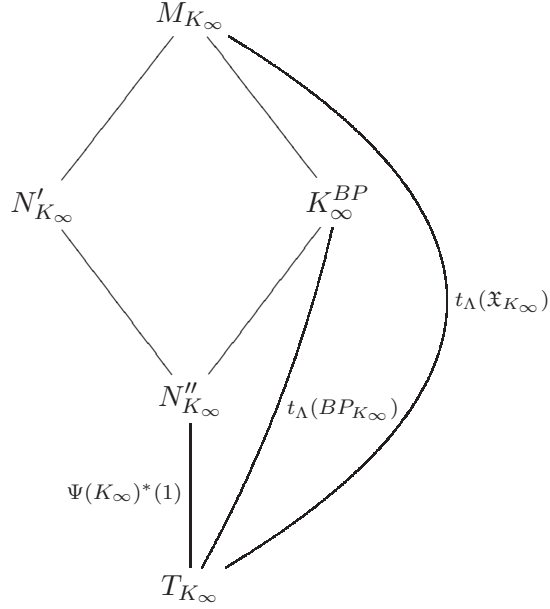
\begin{displaymath}
\xymatrix{
&M_{K_\infty}\ar@{-}[2,1]\ar@{-}[2,-1]\ar@/^8pc/@{-}^{t_\Lambda(\mathfrak{X}_{K_\infty})}[6,0]\\
\\
N'_{K_\infty}\ar@{-}[2,1]&&K^{BP}_\infty\ar@{-}[2,-1]\ar@/^/@{-}^{t_\Lambda(BP_{K_\infty})}[4,-1]\\
\\
&N''_{K_\infty}\ar@{-}_{\Psi(K_\infty)^*(1)}[2,0]\\
\\
&T_{K_\infty}
}
\end{displaymath}

\begin{remark}\label{idea}
As is apparent in the above diagram, if $\Psi(K_\infty)=0$, the restriction sends $\mathrm{Gal}(M_{K_\infty}/N'_{K_\infty})$ isomorphically onto $t_\Lambda(BP_{K_\infty})$. Therefore the triviality of $\Psi(K_{\infty})$ implies that the surjection $t_\Lambda(\mathfrak{X}_{K_\infty})\to t_\Lambda(BP_{K_\infty})$ splits in the category of $\Lambda$-modules. 
\end{remark}

\section{Conditions for the splitting via Galois descent}
Now that the Iwasawa theoretical setting has been described, we give conditions for the splitting of the inclusion $\iota_{k,\infty}:\wkin{k}{(i)} \to \hdin{o'_k}{(i)}$. This map is intimately related to the projection $\pi:t_\Lambda(\mathfrak{X}_{K_\infty})\to t_\Lambda(BP_{K_\infty})$ that appeared in the preceding section, as we shall explain in this section. 

We begin with a refined version of Remark \ref{idea}. We need the following lemma that will also be used in the next section.

\begin{lemma}\label{casa}
Let $R$ be a ring and let
$$
\begin{CD}
@.@.0@.0\\
@.@.@VVV@VVV\\
0@>>>A_1@>\alpha_1>>A_2@>\alpha_2>>A_3@>>>0\\
@.@V\iota_1VV@V\iota_2 VV@V \iota_3 VV\\
0@>>>B_1@>\beta_1>>B_2@>\beta_2>>B_3@>>>0\\
@.@.@V\pi_2VV@V\pi_3VV\\
@.@.T_2@>\tau_2>>T_3\\
@.@.@VVV@VVV\\
@.@.0@.0\\
\end{CD}
$$ 
be a commutative diagram of $R$-modules with exact rows and columns. Suppose that $\tau_2$ (or equivalently $\iota_1$) is an isomorphism. Then the following are equivalent:
\begin{enumerate}
\item[(i)] $\alpha_2$ and $\pi_2$ split;
\item[(ii)] $\beta_2$ and $\pi_3$ split.
\end{enumerate}
\end{lemma}
\begin{proof}
If $\pi_2$ has a section, say $\lambda_2$, then $\beta_2\circ \lambda_2\circ \tau_2^{-1}$ is a section of $\pi_3$. If moreover $\alpha_2$ splits, then $\alpha_1$ has a section, say $\rho_1$. Then $\iota_1\circ\rho_1\circ\nu_2$ is a section of $\beta_1$, where $\nu_2$ is a section of $\iota_2$. In particular $\beta_2$ splits.\\
If $\beta_2$ splits, then $\beta_1$ has a section, say $\sigma_1$. Then $\iota_{1}^{-1}\circ\sigma_1\circ \iota_2$ is a section of $\alpha_1$ (in particular $\alpha_2$ splits).
If moreover $\pi_3$ has a section, say $\lambda_3$, then $\sigma_2\circ\lambda_3\circ\tau_2$ is a section of $\pi_2$, where $\sigma_2$ is a section of $\beta_2$.
\end{proof}

\begin{prop}\label{lmn2}
Suppose that Leopoldt's and Gross's conjectures hold for all the layers $K_n$. Then the following conditions are equivalent:
\begin{itemize}[leftmargin=12pt]
\item the surjective $\Lambda$-morphism $t_\Lambda(\mathfrak{X}_{K_\infty})(-i)_\Delta\to \Psi(K_\infty)^*(1-i)_\Delta$ splits;
\item the surjective $\Lambda$-morphisms $t_{\Lambda}(BP_{K_\infty})(-i)_\Delta\to \Psi(K_\infty)^*(1-i)_\Delta$ and $t_\Lambda(\mathfrak{X}_{K_\infty})(-i)_\Delta\to t_{\Lambda}(BP_{K_\infty})(-i)_\Delta$ split. 
\end{itemize}
In particular, if $\Psi(K_\infty)(i-1)^\Delta=0$, then the surjective $\Lambda$-homomorphism 
$$t_\Lambda(\mathfrak{X}_{K_\infty})(-i)_\Delta\to t_{\Lambda}(BP_{K_\infty})(-i)_\Delta$$ 
splits.
\end{prop}
\begin{proof}
Note that $N'_{K_\infty}/k$ and $K^{BP}_\infty/k$ are Galois extensions. This implies that $\mathrm{Gal}(M_{K_\infty}/N'_{K_\infty})$ and $\mathrm{Gal}(K^{BP}_\infty/N''_{K_\infty})$ are $\Lambda[\Delta]$-submodules of $t_\Lambda(\mathfrak{X}_{K_\infty})$ and $t_\Lambda(BP_{K_\infty})$, respectively. Therefore we have a commutative diagram of $\Lambda[\Delta]$-modules with exact rows and columns
$$
\begin{CD}
@.@.0@.0\\
@.@.@VVV@VVV\\
0@>>>\mathrm{Gal}(M_{K_\infty}/K^{BP}_\infty)@>>>\mathrm{Gal}(M_{K_\infty}/N''_{K_\infty})@>>>\mathrm{Gal}(K^{BP}_\infty/N''_{K_\infty})@>>>0\\
@.@|@VVV@VVV\\
0@>>>\mathrm{Gal}(M_{K_\infty}/K^{BP}_\infty)@>>>t_\Lambda(\mathfrak{X}_{K_\infty})@>>>t_\Lambda(BP_{K_\infty})@>>>0\\
@.@.@VVV@VVV\\
@.@.\Psi(K_\infty)^*(1)@=\Psi(K_\infty)^*(1)@>>>0\\
@.@.@VVV@VVV\\
@.@.0@.0
\end{CD}
$$
Furthermore the surjective homomorphism $\mathrm{Gal}(M_{K_\infty}/N''_{K_\infty})\to\mathrm{Gal}(K^{BP}_\infty/N''_{K_\infty})$ of $\Lambda[\Delta]$-modules splits (since $\mathrm{Gal}(M_{K_\infty}/N'_{K_\infty})$ is an isomorphic preimage of $\mathrm{Gal}(K^{BP}_\infty/N''_{K_\infty})$). Now tensor the above diagram with $\zp(-i)$ and take $\Delta$-coinvariants. The resulting diagram satisfies the hypotheses of Lemma \ref{casa}, which implies the equi\-va\-lence between the two assertions of the proposition.
\end{proof}

Before the description of the relation between $\iota_{k,\infty}$ and the projection $t_\Lambda(\mathfrak{X}_{K_\infty})\to t_\Lambda(BP_{K_\infty})$, we quote the following well-known lemma.

\begin{lemma}\label{ltzpt}
Let $M$ be a $\Lambda[\Delta]$-module which is a finitely generated $\Lambda$-module. Suppose that $(t_\Lambda(M))_{\Gamma_n}$ is finite for all $n\in \mathbb{N}$. Then the natural map $M\to \lim\limits_{\longleftarrow}M_{\Gamma_n}$ induces an isomorphism of $\Lambda[\Delta]$-modules
$$t_\Lambda(M)\cong \lim\limits_{\longleftarrow}t_{\zp}(M_{\Gamma_n}).$$
\end{lemma}
\begin{proof}
See \cite[proof of Proposition 2.1]{N3}.
\end{proof}

The proof of the following proposition is inspired by the papers \cite{N1, N2}.

\begin{prop}\label{nqd}
Suppose that  Leopoldt's conjecture holds for all the layers $K_n$. Then, for every $i\geq 0$ and $i\ne 1$, there are natural isomorphisms of $\Lambda[\Delta]$-modules
$$t_{\Lambda}(\mathfrak{X}_{K_\infty})(-i)\cong \hdin{o'_K}{(i)}^*\quad \textrm{and} \quad t_{\Lambda}(BP_{K_\infty})(-i)\cong \wkin{K}{(i)}^*,$$ 
fitting into a commutative diagram 
$$
\begin{CD}
\hdin{o'_K}{(i)}^*@>\iota_{K,\infty}^*>>\wkin{K}{(i)}^*\\
@V\wr VV@V\wr VV\\
t_{\Lambda}(\mathfrak{X}_{K_\infty})(-i)@>\pi>>t_{\Lambda}(BP_{K_\infty})(-i),
\end{CD}
$$
where $\pi$ is induced by the natural projection $\mathfrak{X}_{K_\infty}\to BP_{K_\infty}$.
\end{prop}
\begin{proof}
We shall prove the statement supposing that $K_{n_0}=K$, \textit{i.e.} all $p$-adic primes are totally ramified in $K_\infty/K$. The general case easily follows since no object appearing in the statement changes if $K$ is replaced by a higher layer of the cyclotomic tower. Furthermore every morphism we shall consider in the proof can be easily checked to be invariant with respect to the action of the Galois group of any Galois extension $K/E$. 

Recall that, for every $i\in\mathbb{Z}$, there is an isomorphism of $\zp[\Delta]$-modules
\begin{equation}\label{h2t}
\hdet{o'_K}{(i)}^*\cong t_{\zp}(\mathfrak{X}_{K_\infty}(-i)_{\Gamma})
\end{equation}
(see \cite[Lemme 4.1]{N3}). Similarly one shows that for every prime $v$ above $p$ in $K$ there is an isomorphism of $\zp[\Delta_v]$-modules
$$\hdet{K_v}{(i)}^*\cong t_{\zp}(\mathfrak{X}_{K_{v,_\infty}}(-i)_{\Gamma_v}),$$
where $\mathfrak{X}_{K_{v,\infty}}$ is the Galois group of the maximal abelian pro-$p$-extension of the cyclotomic $\zp$-extension $K_{v,\infty}/K_v$, $\Gamma_v=\mathrm{Gal}(K_{v,\infty}/K_v)$ and $\Delta_v=\mathrm{Gal}(K_v/k_{v_0})$ where $v_0$ is the prime of $k$ below $v$.
The above isomorphisms yield a commutative diagram of $\zp[\Delta]$-modules
\begin{equation}\label{diag}
\begin{CD}
\underset{v\mid p}{\oplus}\hdet{K_v}{(i)}^*@>>>\hdet{o'_K}{(i)}^*\\
@V\wr VV@V\wr VV\\
\underset{v\mid p}{\oplus}t_{\zp}(\mathfrak{X}_{K_{v,\infty}}(-i)_{\Gamma_v})@>>>t_{\zp}(\mathfrak{X}_{K_\infty}(-i)_{\Gamma}).
\end{CD}
\end{equation}
If $i\ne 1$, then the group $\left(t_{\Lambda_v}(\mathfrak{X}_{K_{v,\infty}}(-i)\right)_{\Gamma_v}$ is finite for every $v\mid p$ in $K$, where $\Lambda_v=\zp[\![\Gamma_v]\!]$ (this follows from \cite[Theorem 25]{Iw}). The same holds for $\left(t_{\Lambda}(\mathfrak{X}_{K_\infty})(-i)\right)_{\Gamma}$ provided $i\geq 2$ (this follows easily from \cite[Lemme 4.3]{N3}) or $i=0$ and Leopoldt's conjecture holds for $K$ (see \cite[Lemma 21]{Iw}). Moreover, an easy check shows the diagram in (\ref{diag}) to be compatible with the restriction maps on both sides. Therefore, for every $i\geq 0$ and $i\ne 1$, taking direct limits over $K_\infty/K$ with respect to those maps and using Lemma \ref{ltzpt}, we get a commutative diagram (under Leopoldt's conjecture for all the layers $K_n$ when $i=0$)
$$
\begin{CD}
\underset{v\mid p}{\oplus}\hdin{K_v}{(i)}^*@>>>\hdin{o'_{K}}{(i)}^*\\
@V\wr VV@V\wr VV\\
\underset{v\mid p}{\oplus}t_{\Lambda_v}(\mathfrak{X}_{K_{v,\infty}}(-i))@>>>t_{\Lambda}(\mathfrak{X}_{K_\infty}(-i)),
\end{CD}
$$
where $\hdin{K_v}{(i)}=\lim\limits_{\longrightarrow}\hdet{(K_v)_n}{(i)}$. Since clearly
\begin{equation}\label{tinv}
t_{\Lambda}(\mathfrak{X}_{K_\infty}(-i))=t_{\Lambda}(\mathfrak{X}_{K_\infty})(-i)\quad\textrm{and} \quad t_{\Lambda_v}(\mathfrak{X}_{K_{v,\infty}}(-i))=t_{\Lambda_v}(\mathfrak{X}_{K_{v,\infty}})(-i),
\end{equation}
we deduce that, for every $i\geq 2$, there is a commutative diagram
\begin{equation}\label{h2inv}\begin{CD}
\hdin{o'_{K}}{(i)}@>>>\underset{v\mid p}{\oplus}\hdin{K_v}{(i)}\\
@V\wr VV@V\wr VV\\
\hdin{o'_{K}}{}(i)@>>>\underset{v\mid p}{\oplus}\hdin{K_v}{}(i)\\
\end{CD}
\end{equation}
which implies in particular that 
\begin{equation}\label{wkinv}
\wkin{o'_K}{(i)}\cong \wkin{o'_K}{}(i).
\end{equation}
Now, the field of Bertrandias-Payan over $K$ certainly contains the compositum of all the $\zp$-extensions of $K$. Therefore the map $t_{\zp}(\mathfrak{X}_{K})\to t_{\zp}(BP_{K})$, induced by (\ref{XBP}), is surjective. In fact, it follows easily from the arguments of \cite[Th\'eor\`eme 4.2]{N1} that there is a commutative diagram of $\zp[\mathrm{Gal}(K/k)]$-modules with surjective rows
$$
\begin{CD}
\hdet{o'_{K}}{}^*@>>>\wket{K}{}^*\\
@V\wr VV@V\wr VV\\
 t_{\zp}(\mathfrak{X}_{K}) @>>> t_{\zp}(BP_{K})\\
\end{CD}
$$
where the left-hand isomorphism is (\ref{h2t}), while the right-hand one is induced by the diagram. As before, the above diagram is compatible with restriction maps and taking inverse limits with respect to those maps over $K_\infty/K$ yields a commutative diagram 
\begin{equation}\label{coomgalois}
\begin{CD}
\hdin{o'_K}{}^*@>>> \wkin{o'_K}{}^*\\
@V\wr VV@V\wr VV\\
 \lim\limits_{\longleftarrow} t_{\zp}(\mathfrak{X}_{K_n})@>>>\lim\limits_{\longleftarrow} t_{\zp}(BP_{K_n}).\\
\end{CD}
\end{equation}
If Leopolt's conjecture holds for $K_n$ for every $n\in\mathbb{N}$, using Lemma \ref{ltzpt} we can replace the bottom line of diagram (\ref{coomgalois}) by the surjection $t_{\Lambda}(\mathfrak{X}_{K_\infty})\to t_{\Lambda}(BP_{K_\infty})$. Then the result follows by applying the functor $\otimes_{\zp}\zp(-i)$ to diagram (\ref{coomgalois}) and using (\ref{tinv}), (\ref{h2inv}) and (\ref{wkinv}).
\end{proof}

\begin{remark}\label{chazadinfty}
From the proof of the above proposition, we deduce that, if $i\geq 2$ and Leopoldt's conjecture holds for $K_n$ for all $n\in\mathbb{N}$, then there is a commutative diagram of $\Lambda[\Delta]$-modules
$$
\begin{CD}
\wkin{K}{(i)} @>\iota_{K,\infty,i}>> \hdin{o'_K}{(i)}\\
@V\wr VV@V\wr VV\\
\wkin{K}{}(i) @>\iota_{K,\infty,0}(i)>> \hdin{o'_K}{}(i)\\
\end{CD}
$$
whose vertical arrows are natural isomorphisms. 
\end{remark}

\begin{cor}\label{ideaNQD}
Suppose that Leopoldt's conjecture holds for all the layers $K_n$. Then for every $i\geq 2$ the following conditions are equivalent:
\begin{enumerate}
\item the injective $\Lambda$-homomorphism $\wkin{k}{(i)}\to \hdin{o'_k}{(i)}$ splits;
\item the surjective $\Lambda$-homomorphism $t_\Lambda(\mathfrak{X}_{K_\infty})(-i)_\Delta\to t_{\Lambda}(BP_{K_\infty})(-i)_\Delta$ splits.
\end{enumerate}
In particular, if $\Psi(K_\infty)(i-1)^\Delta=0$, the inclusion $\wkin{k}{(i)}\to \hdin{o'_k}{(i)}$ splits as soon as Gross's conjecture holds for all the layers $K_n$.
\end{cor}
\begin{proof}
Using standard properties of Pontryagin duality and the equality $\hdin{o'_K}{(i)}^\Delta=\hdin{o'_k}{(i)}$, it is easy to show that the first condition is equivalent to the splitting of the surjective map
$$\left(\hdin{o'_K}{(i)}^*\right)_\Delta\to \left(\wkin{K}{(i)})^*\right)_\Delta.$$
Therefore the equivalence between the two statements of the corollary follows by Proposition \ref{nqd}. The last assertion follows by Proposition \ref{lmn2}.
\end{proof} 

Once we have given conditions for the inclusion $\wkin{k}{(i)}\subseteq \hdin{o'_k}{(i)}$ to split, we aim to see when the splitting at infinite level implies the splitting at finite levels. It is natural to look at this problem using Galois descent.

\begin{prop}\label{main}
Let $k$ be a number field and $i\geq 2$ be an integer. Suppose that
\begin{enumerate}  
\item[(i)] the surjective homomorphism $\hdet{o'_k}{(i)}\to \hdin{o'_k}{(i)}^{\Gamma}$ splits;
\item[(ii)] the injective $\Lambda$-homomorphism $\wkin{k}{(i)}\to \hdin{o'_k}{(i)}$ splits.
\end{enumerate}
Then the inclusion $\wket{k}{(i)}\to \hdet{o'_k}{(i)}$ splits.
\end{prop}
\begin{proof} 
If $\wkin{k}{(i)} \to \hdin{o'_k}{(i)}$ splits in the category of $\Lambda$-modules, then
$\wkin{k}{(i)}^{\Gamma}\to \hdin{o'_k}{(i)}^{\Gamma}$ splits. The latter condition implies the splitting of $\wket{k}{(i)}\to \hdet{o'_k}{(i)}$, thanks to (i), Theorem \ref{descent} and Lemma \ref{casa}. 
\end{proof}

Combining Proposition \ref{main} and Corollary \ref{ideaNQD} gives the main result of this section.

\begin{teo}\label{randonnee}
Let $k$ be a number field and $i\geq 2$ be an integer. Suppose that
\begin{enumerate}
\item[(a)] the surjective homomorphism $\hdet{o'_k}{(i)}\to \hdin{o'_k}{(i)}^{\Gamma}$ splits; 
\item[(b)] $\Psi(K_\infty)(i-1)^{\Delta}=0$;
\item[(c)] Gross's and Leopoldt's conjectures hold for all the layers $K_n$. 
\end{enumerate}
Then the injective homomorphism $\wket{k}{(i)}\to \hdet{o'_k}{(i)}$ splits. 
\end{teo}

\begin{remarks}
\begin{enumerate}
\item[1)] Suppose that the Iwasawa $\mu$ invariant of $X'_{K_\infty}$ is trivial (this holds for example if $k$ is abelian by a classical result of Ferrero and Washington). Then, by Theorem \ref{descent}, Condition (a) of Theorem \ref{randonnee} holds asymptotically, \textit{i.e.} the surjection $\hdet{o'_{k_n}}{(i)}\to \hdin{o'_k}{(i)}^{\Gamma_n}$ splits for $n$ large enough. 
\item[2)] Suppose that $k$ is totally real (hence $K$ is CM).
\begin{itemize}
\item If $i$ is even, then $\xKo(i-1)_\Delta=0$ by Proposition \ref{coates}. In particular  Condition (a) of Theorem \ref{randonnee} holds, thanks to Theorem \ref{descent}.
\item If $i$ is odd, then we have
$$\Psi(K_{\infty})(i-1)^{\Delta}\subseteq \Psi(K_{\infty})(i-1)^{\Delta_0}\cong \Psi(K_{\infty})^{\Delta_0}(i-1),$$
where $K^+$ denotes the maximal totally real subfield of $K$ and $\Delta_0=\mathrm{Gal}(K/K^+)\subseteq \Delta$. Now, under Greenberg's conjecture $\Psi(K_\infty)^{\Delta_0}=\Psi(K^+_\infty)=0$ by Remark \ref{oujda}. In particular, under Greenberg's conjecture, Condition (b) of Theorem \ref{randonnee} holds. 
\end{itemize}
\item[3)] Of course, if $\xKo(i-1)_{\Delta}=0$, then, by Theorem \ref{descent}, Condition (a) of Theorem \ref{randonnee} is satisfied. However the condition $\xKo(i-1)_{\Delta}=0$ could in general be too strong for our purposes.  In fact, by Schneider's isomorphism
$$\wket{k}{(i)} \cong X'_{K_\infty}(i-1)_{\mathcal{G}}=(X'_{K_\infty}(i-1)_{\Delta})_\Gamma.$$
In particular, by Nakayama's lemma, $\wket{k}{(i)}$ is trivial precisely when $X'_{K_\infty}(i-1)_{\Delta}$ is. Now, if $k$ is a totally real field satisfying Greenberg's conjecture and $i\equiv 1 \pmod{\#\Delta}$, then 
$$X'_{K_\infty}(i-1)_{\Delta}\cong X'_{k_\infty}=\xko\cong\xKo(i-1)_{\Delta}.$$
Therefore, in this situation, the triviality of $\xKo(i-1)_{\Delta}$ is equivalent to that of $\wket{k}{(i)}$.    
\end{enumerate}
\end{remarks}


\section{Conditions for the splitting via Galois codescent}\label{code}
In this section we shall follow a different approach to find a condition for the inclusion $\wket{k}{(i)}\to\hdet{k}{(i)}$ to split. We study the splitting of the inclusion of $\Lambda$-modules $\wkiw{k}{(i)}\to\hdiw{k}{(i)}$ and then use the good codescent properties of higher wild and tame kernels.
   
Taking inverse limits of the exact sequence (\ref{sequence}) for $i\geq 1$ along the cyclotomic tower, we obtain the following exact sequence of $\Lambda$-modules
\begin{equation}\label{seiw}
0\to\wkiw{k}{(i)}\stackrel{\iota_{k,Iw}}{\longrightarrow} \hdiw{o'_k}{(i)} \stackrel{\pi_{k,Iw}}{\longrightarrow} \underset{v\mid p}{\oplus}\underset{w\mid v}{\oplus}\hdiw{k_v}{(i)}\to \hzin{k}{(1-i)}^*\to 0
\end{equation}
where, in the second sum, $w$ runs over the primes of $k_\infty$ dividing a $p$-adic prime $v$ of $k$. Set  
$$\oplus_{k,Iw}=\underset{v\mid p}{\oplus}\underset{w\mid v}{\oplus}\hdiw{k_v}{(i)},\qquad\tilde\oplus_{k,Iw}=\mathrm{ker}\left(\oplus_{k,Iw}\to \hzin{k}{(1-i)}^*\right),$$
%
%
%
%
so that we have an exact sequence 
\begin{equation}\label{seiwsimp}
0\to\wkiw{k}{(i)}\stackrel{\iota_{k,Iw}}{\longrightarrow}\hdiw{o'_k}{(i)} \stackrel{\pi_{k,Iw}}{\longrightarrow} \tilde\oplus_{k,Iw}\to 0.
\end{equation}
Note that the above exact sequence splits in the category of $\zp$-modules since $\tilde\oplus_{k,Iw}$ is either trivial or a free $\zp$-module ($\hdiw{k_v}{(i)}\cong\hzin{k_v}{(1-i)}^*$ is either trivial or $\zp$-free of rank $1$).

Let $M,N$ be two finitely generated compact $\Lambda$-modules, then we let $\mathrm{Hom}(M, N)$ (resp. $\mathrm{Hom}_\Lambda(M, N)$) denote the group of continuous $\zp$-homomorphisms (resp. $\Lambda$-homomorphisms). Then $\mathrm{Hom}(M, N)$ has a structure of $\Gamma$-module such that $(\gamma f)(m)=\gamma f(\gamma^{-1}m)$, for every $f\in \mathrm{Hom}(M, N)$, $\gamma\in\Gamma$ and $m\in M$. In particular $\mathrm{Hom}(M, N)^{\Gamma}=\mathrm{Hom}_\Gamma(M, N)=\mathrm{Hom}_\Lambda(M, N)$.

\begin{lemma}\label{pavia}
Let 
\begin{equation}\label{m123}
0\to M_1 \stackrel{\iota}{\longrightarrow} M_2 \stackrel{\pi}{\longrightarrow} M_3\to 0
\end{equation} 
be an exact sequence of finitely generated compact $\Lambda$-modules which splits in the category of $\zp$-modules. Then (\ref{m123}) splits in the category of $\Lambda$-modules if and only if 
\begin{equation}\label{condition123}
\mathrm{ker}\Big(\mathrm{Hom}(M_3, M_1)_\Gamma\stackrel{\iota\scriptscriptstyle{\circ}}{\longrightarrow}\mathrm{Hom}(M_3, M_2)_\Gamma\Big)=0.
\end{equation}
\end{lemma}
\begin{proof}
Since (\ref{m123}) splits as $\zp$-modules, we have an exact sequence of $\Gamma$-modules
$$0\to \mathrm{Hom}(M_3, M_1)\stackrel{\iota\scriptscriptstyle{\circ}}{\longrightarrow}\mathrm{Hom}(M_3, M_2)\stackrel{\pi\scriptscriptstyle{\circ}}{\longrightarrow}\mathrm{Hom}(M_3, M_3)\to 0.$$
Note that if (\ref{m123}) splits as $\Gamma$-modules, then clearly the same holds for the above sequence and in particular (\ref{condition123}) holds. 
Suppose conversely that (\ref{condition123}) holds. A standard application of the snake lemma to the above exact sequence gives an exact sequence
$$\mathrm{Hom}_\Gamma(M_3, M_2)\stackrel{\pi\scriptscriptstyle{\circ}}{\longrightarrow}\mathrm{Hom}_\Gamma(M_3, M_3) \to \mathrm{Hom}(M_3, M_1)_\Gamma\stackrel{\iota\scriptscriptstyle{\circ}}{\longrightarrow}\mathrm{Hom}(M_3, M_2)_\Gamma.$$
We deduce that $\pi{\scriptscriptstyle\circ}:\mathrm{Hom}_\Gamma(M_3, M_2)\to \mathrm{Hom}_\Gamma(M_3, M_3)$ is surjective. In particular, there exists $\sigma\in \mathrm{Hom}_\Gamma(M_3, M_2)=\mathrm{Hom}_\Lambda(M_3, M_2)$ such that $\pi{\scriptscriptstyle\circ\,}\sigma=\mathrm{id}_{M_3}$. 
\end{proof}

When $i=1$ we have $\wket{K}{(1)}\cong A'_K$ and therefore (\ref{sequence}) reads 
$$0\to A'_K\to \hdet{o'_K}{(1)}\to \underset{v\mid p}{\oplus} \hdet{K_v}{(1)}\to \hzet{K}{}^*\to 0.$$
Note that 
$$\hdet{K_v}{(1)}\cong  \hzet{K_v}{}^*\cong \zp$$
by local duality. Hence the inclusion $A'_K\to \hdet{o'_K}{(1)}$ splits as $\zp[\Delta]$-modules and therefore also as $\Lambda[\Delta]$-modules (with trivial $\Gamma$-action), since it maps $A'_K$ isomorphically onto $t_{\zp}(\hdet{o'_K}{(1)})$. In particular the inclusion $A'_K(i-1)_\Delta\to \hdet{o'_K}{(1)}(i-1)_\Delta$ splits as $\Lambda$-modules.


\begin{lemma}\label{chazad}
For every $i\in\mathbb{Z}$, there is a commutative diagram of $\Lambda[\Delta]$-modules
$$
\begin{CD}
\wkiw{K}{(i)} @>\iota_{K,Iw,i}>> \hdiw{o'_K}{(i)}\\ 
@V\wr VV@V\wr VV\\ 
\wkiw{K}{}(i) @>\iota_{K,Iw,0}(i)>> \hdiw{o'_K}{}(i)\\ 
\end{CD}
$$
whose vertical arrows are natural isomorphisms. 
\end{lemma}
\begin{proof}
We have   
\begin{eqnarray*}
\hdiw{o'_K}{(i)}&=&\lim\limits_{\underset{n}{\longleftarrow}} \lim\limits_{\underset{m}{\longleftarrow}}\hdetp{o'_{K_n}}{m}{(i)}\\
&=&\lim\limits_{\underset{m}{\longleftarrow}} \lim\limits_{\underset{n\geq m}{\longleftarrow}}\hdetp{o'_{K_n}}{m}{(i)}\\ 
&\cong&\lim\limits_{\underset{m}{\longleftarrow}} \lim\limits_{\underset{n\geq m}{\longleftarrow}}\hdetp{o'_{K_n}}{m}{}(i)\quad\textrm{(since $\mu_{p^m}\subseteq K_n$ if $n\geq m$)}\\
&=&\lim\limits_{\underset{n}{\longleftarrow}} \lim\limits_{\underset{m}{\longleftarrow}}\hdetp{o'_{K_n}}{m}{}(i)=\hdiw{o'_K}{}(i).
\end{eqnarray*}
A similar argument shows that, for any $p$-adic prime $w$ of $K$, there is an isomorphism $\hdiw{K_w}{(i)}\cong \hdiw{K_w}{}(i)$ which is compatible with the one above. Hence the statement of the lemma follows. 
\end{proof}

\begin{remark}
If the Iwasawa $\mu$-invariant of $X'_{K_\infty}$ is trivial and $i\geq 2$, there exist isomorphisms of $\Lambda[\Delta]$-mo\-du\-les 
$$\wkin{K}{(i)}\cong \wkiw{K}{(i)}\otimes \qp/\zp,\quad\hdin{o'_K}{(i)}\cong \hdiw{o'_K}{(i)}\otimes \qp/\zp$$
(see \cite[Lemma 2.4]{KM2}). Thus Lemma \ref{chazad} can be used to prove the result of Remark \ref{chazadinfty}.
\end{remark}

We are now ready for the main result of this section.

\begin{teo}\label{main2}
Let $k$ be a number field and $i\geq 2$ be an integer. Suppose that the map 
\begin{equation}\label{mappa}
\alpha:\mathrm{Hom}(\tilde\oplus_{k,Iw},\xK(i-1)_\Delta)_\Gamma \to \mathrm{Hom}(\tilde\oplus_{k,Iw},A'_K(i-1)_\Delta)_\Gamma,
\end{equation} 
induced by the natural map $\xK(i-1)_\Delta\to A'_K(i-1)_\Delta$, is injective.
Then the injective homomorphism $\wket{k}{(i)}\to \hdet{o'_k}{(i)}$ splits.
\end{teo}
\begin{proof}
By (\ref{codescentwh}), it suffices to show that the exact sequence (\ref{seiwsimp}) splits in the category of $\Lambda$-modules.
Note that 
$$\wkiw{k}{(i)}\cong \xK(i-1)_\Delta,$$
by Schneider's isomorphism. Then, according to Lemma \ref{pavia}, the exact sequence (\ref{seiwsimp}) splits if and only if the map 
\begin{equation}\label{remplace}
\beta:\mathrm{Hom}(\tilde\oplus_{k,Iw},\xK(i-1)_\Delta)_\Gamma\to\mathrm{Hom}\left(\tilde\oplus_{k,Iw}, \hdiw{o'_k}{(i)}\right)_\Gamma,
\end{equation}
induced by $\iota_{K,Iw}$ via the above isomorphism, is injective. Hence it is sufficient to show that $\mathrm{ker}\left(\beta\right)\subseteq \mathrm{ker}\left(\alpha\right)$.

Using Lemma \ref{chazad}, we get an isomorphism $\hdiw{o'_k}{(i)}\cong\hdiw{o'_K}{(1)}(i-1)_\Delta$ and therefore a morphism $\hdiw{o'_k}{(i)}\to \hdet{o'_K}{(1)}(i-1)_{\Delta}$ fitting into a commutative diagram of $\Lambda$-modules
$$
\begin{CD}
\xK(i-1)_\Delta @>>> A'_K(i-1)_\Delta\\
@VVV@VVV\\
\hdiw{o'_k}{(i)} @>>>\hdet{o'_K}{(1)}(i-1)_{\Delta}
\end{CD}
$$
whose vertical arrows are injective. Applying the functor $\mathrm{Hom}(\tilde\oplus_{k,Iw},\--)_\Gamma$ to the above diagram, we get a commutative diagram 
$$
\begin{CD}
\mathrm{Hom}(\tilde\oplus_{k,Iw},\xK(i-1)_\Delta)_\Gamma @>\alpha>> \mathrm{Hom}(\tilde\oplus_{k,Iw},A'_K(i-1)_\Delta)_\Gamma\\
@V\beta VV@VVV\\
\mathrm{Hom}(\tilde\oplus_{k,Iw},\hdiw{o'_k}{(i)})_\Gamma @>>>\mathrm{Hom}(\tilde\oplus_{k,Iw},\hdet{o'_K}{(1)}(i-1)_{\Delta})_\Gamma.
\end{CD}
$$
According to the discussion prior to Lemma \ref{pavia}, the injection $A'_K(i-1)_\Delta\to \hdet{o'_K}{(1)}(i-1)_\Delta$ splits as a homomorphism of $\Lambda$-modules. Therefore the same holds for the injection $\mathrm{Hom}(\tilde\oplus_{k,Iw},A'_K(i-1)_\Delta)\to\mathrm{Hom}(\tilde\oplus_{k,Iw},\hdet{o'_K}{(1)}(i-1)_{\Delta})$. Hence the right vertical arrow of the above diagram is injective and we have $\mathrm{ker}\left(\beta\right)\subseteq \mathrm{ker}\left(\alpha\right)$, as desired.
\end{proof}


The condition of Theorem \ref{main2}, despite its aspect, can be simplified in several concrete situations, as the following result shows.

\begin{cor}\label{simplecase}
Suppose that $n_0(K)=0$ (\textit{i.e.} all $p$-adic places are totally ramified in $K_\infty/K$). Then, if $\Psi(K)(i-1)^\Delta=0$, the inclusion $\wket{k}{(i)}\to \hdet{o'_k}{(i)}$ splits.
\end{cor}
\begin{proof}

Since $Q:=\hzin{k}{(1-i)}^*$ is either trivial or a free $\zp$-module, the exact sequence of $\Lambda$-modules 
$$0\to \tilde\oplus_{k,Iw}\to\oplus_{k,Iw}\to Q\to 0$$
splits in the category of $\zp$-modules. We thus get a commutative diagram of $\Lambda$-modules with exact rows
$$
\xymatrix{
0\ar[r]&\mathrm{Hom}(Q,\xK(i-1)_\Delta) \ar[d] \ar[r] & \mathrm{Hom}(\oplus_{k,Iw},\xK(i-1)_\Delta)\ar[d] \ar[r] &\mathrm{Hom}(\tilde\oplus_{k,Iw},\xK(i-1)_\Delta)\ar[d]\ar[r]&0\,\,\\
0\ar[r]&\mathrm{Hom}(Q,A'_K(i-1)_\Delta)\ar[r] & \mathrm{Hom}(\oplus_{k,Iw},A'_K(i-1)_\Delta) \ar[r] &\mathrm{Hom}(\tilde\oplus_{k,Iw},A'_K(i-1)_\Delta)\ar[r]&0.\\}
$$
Now note that
\begin{eqnarray*}
Q=\hzin{k}{(1-i)}^*&\cong&\left\{\begin{array}{ll}\zp(i-1)&\textrm{if $i\equiv 1 \pmod{[K:k]}$}\\0&\textrm{otherwise,}\end{array}\right. \quad \textrm{as $\Gamma$-modules,}\\
\hdiw{k_v}{(i)}&\cong&\left\{\begin{array}{ll}\zp(i-1)&\textrm{if $i\equiv 1 \pmod{[k_v(\mu_p):k_v]}$}\\0&\textrm{otherwise,}\end{array}\right. \quad \textrm{as $\Gamma_v$-modules,}
\end{eqnarray*}
where $\Gamma_v=\mathrm{Gal}(k_{v,\infty}/k_v)$ is the Galois group of the cyclotomic $\zp$-extension of $k_v$. Since, by hypothesis, all $p$-adic places are totally ramified in $K_\infty/K$ (and hence in $k_\infty/k$), we have $\Gamma=\Gamma_v$ and an isomorphism of $\Lambda$-modules
$$\oplus_{k,Iw}\cong \underset{v\mid p}{\oplus^\prime}\zp(i-1),$$
where the module on the right-hand side is the sum of the $\Lambda$-module $\zp(i-1)$ over the $p$-adic places $v$ of $k$ such that $i\equiv 1 \pmod{[k_v(\mu_p):k_v]}$.

For every $\Lambda[\Delta]$-module $M$ and every $j\in \mathbb{Z}$, there is a natural isomorphism of $\Lambda[\Delta]$-modules $M(-j)\cong \mathrm{Hom}(\zp(j),M)$. Moreover, $M(j)_\Delta(-j)\cong M^{[-j]}$ as $\Lambda[\Delta]$-modules. As usual, here $M^{[-j]}$ denotes the eigenspace of $M$ where $\Delta$ acts as multiplication by $\omega^{-j}$, $\omega:\Delta\to \zp^\times$ being the Teichm\"uller character. In particular, if $i\equiv 1 \pmod{[K:k]}$, then obviously $i\equiv 1 \pmod{[k_v(\mu_p):k_v]}$ for every $p$-adic place $v$ and from the above diagram we deduce a commutative diagram of $\Lambda$-modules with exact rows
$$
\begin{CD}
0@>>>\xKj{1-i} @>>> \underset{v\mid p}{\oplus^\prime}\xKj{1-i} @>>>\mathrm{Hom}(\tilde\oplus_{k,Iw},\xK(i-1)_\Delta)@>>>0\,\,\\
@.@VVV@VVV@VVV\\
0@>>>\aKj{1-i} @>>> \underset{v\mid p}{\oplus^\prime}\aKj{1-i}@>>>\mathrm{Hom}(\tilde\oplus_{k,Iw},A'_K(i-1)_\Delta)@>>>0.
\end{CD}
$$
If instead $i\not \equiv 1 \pmod{[K:k]}$, the terms in the left-hand column are trivial. We shall continue the proof in the case where $i\equiv 1 \pmod{[K:k]}$, the case $i\not\equiv 1 \pmod{[K:k]}$ follows by similar arguments. Of course if $i\equiv 1 \pmod{[K:k]}$, then $\oplus^\prime_{v\mid p}\zp(i-1)=\oplus_{v\mid p}\zp(i-1)$.
Then taking $\Gamma$-coinvariants in the above diagram we get a commutative diagram of $\zp$-modules with exact rows
$$
\begin{CD}
@.(\xKj{1-i})_\Gamma @>>> \underset{v\mid p}{\oplus}(\xKj{1-i})_\Gamma @>>>\mathrm{Hom}(\tilde\oplus_{k,Iw},\xK(i-1)_\Delta)_\Gamma @>>>0\,\,\\
@.@VVV@VVV@VV\alpha V\\
0@>>>\aKj{1-i} @>>> \underset{v\mid p}{\oplus}\aKj{1-i} @>>>\mathrm{Hom}(\tilde\oplus_{k,Iw},A'_K(i-1)_\Delta)_\Gamma @>>>0,
\end{CD}
$$
where $\alpha$ is the map of Theorem \ref{main2}. 
Note that $\ker\left((\xKj{1-i})_\Gamma\to \aKj{1-i}\right)=\Psi(K)^{[1-i]}=0$ by hypothesis.
Hence the snake lemma gives an inclusion
$$\ker\left(\alpha
\right)\hookrightarrow \mathrm{coker}\left((\xKj{1-i})_{\Gamma}\to \aKj{1-i} \right).$$
Since $n_0(K)=0$, the map $\xK\to A'_K$ is surjective so the corollary follows by Theorem \ref{main2}. 
\end{proof}

\begin{remark}
It is interesting to compare the hypotheses of Theorem \ref{randonnee} with those of Corollary \ref{simplecase} (and hence of Theorem \ref{main2}). If $\Psi(K)(i-1)^\Delta=0$, then of course $\Psi(K_\infty)(i-1)^\Delta=0$ by Lemma \ref{CPsi}. On the other hand, even if $\Psi(K_\infty)(i-1)^\Delta=0$ and $\hdet{o'_k}{(i)}\to \hdin{o'_k}{(i)}^\Gamma$ splits, $\Psi(K)(i-1)^\Delta$ may be nontrivial. For example, consider the case where $p=3$, $i=2$ and $k=\mathbb{Q}(\sqrt{-3\cdot 3739})$ (for which $n_0(K)=0$). Then the PARI program developed by Browkin and Gangl (see \cite{BG}) gives $\#H^2_{\acute{e}t}(o'_k,\mathbb{Z}_3(2))=3^2$. Moreover $k$ has one prime $v$ above $3$ and $k_v=\mathbb{Q}_3(\mu_3)$. In particular there is an exact sequence
$$0\to \left(A'_K/3(1)\right)_\Delta\to H^2_{\acute{e}t}(o'_k,\mathbb{Z}_3(2))/3\to \mathbb{Z}/3\mathbb{Z}\to 0$$
(see \cite[proof of Theorem 6.6]{Ke} or \cite[Theorem 6.2]{Ta2}). Using PARI, we get $\#\left(A'_K/3(1)\right)_\Delta=3$ and therefore $H^2_{\acute{e}t}(o'_k,\mathbb{Z}_3(2))$ has $3$-rank $2$, giving $H^2_{\acute{e}t}(o'_k,\mathbb{Z}_3(2))\cong \mathbb{Z}/3\times\mathbb{Z}/3\mathbb{Z}$. This immediately implies that the surjection $H^2_{\acute{e}t}(o'_k,\mathbb{Z}_3(2))\to H^2_{\infty}(o'_k,\mathbb{Z}_3(2))^\Gamma$ splits. Furthermore $\Psi(K_\infty)(i-1)^\Delta\cong\Psi(k'_\infty)(i-1)$, where $k'=\mathbb{Q}(\sqrt{3739})$ is the totally real subfield of $K=k(\mu_3)$. In particular $\Psi(K_\infty)(i-1)^\Delta=0$ by Remark \ref{oujda}. On the other hand, using again PARI and the methods of \cite[Section 4.3]{LMN}, we get $\Psi(K)(i-1)^\Delta\cong\Psi(k')(i-1)\ne 0$. Note that this also shows that the converse of Corollary \ref{simplecase} does not hold in general.   
\end{remark}

\section{A general criterion using class groups}\label{crit}
In this section we will prove a criterion for the inclusion $\wket{k}{(i)}\to \hdet{o'_k}{(i)}$ to split in terms of the triviality of some codescent kernels which are somehow reminiscent of the $\Psi(K_n)$. In fact we shall use this criterion to give a different proof of Corollary \ref{simplecase}.

For every $i\geq 1$ and every $m\in \mathbb{N}$, let $k(\mu_{p^m}^{\otimes i})$ be the subfield of $k(\mu_{p^m})$ which is fixed by the kernel of the homomorphism $\mathrm{Gal}(k(\mu_{p^m})/k)\to \mathrm{Aut}(\mu_{p^m}^{\otimes i})$, induced by the action of $\mathrm{Gal}(k(\mu_{p^m})/k)$ on $\mu_{p^m}^{\otimes i}$. 
For this section, we shall set $j=i-1$. We also recall the notation $\mathcal{G}_n=\mathrm{Gal}(K_\infty/k_n)$ (and $\mathcal{G}_0=\mathcal{G}$) and we set
$$\wketp{k}{m}{(i)}=\ker\left(\hdetp{o'_k}{m}{(i)}\to \underset{v\mid p}{\oplus} \hdetp{k_v}{m}{(i)}\right).$$

The following well-known lemma indicates the strategy of our criterion.

\begin{lemma}\label{splitpure}
The following conditions are equivalent: 
\begin{itemize}
\item $\wket{k}{(i)}\to \hdet{o'_k}{(i)}$ splits;
\item for every $m\in\mathbb{N}$, the map $\wket{k}{(i)}/p^m\to \hdet{o'_k}{(i)}/p^m$ is injective.
\end{itemize}
\end{lemma}
\begin{proof}
The second condition can be rephrased by saying that $\wket{k}{(i)}$ is pure in $\hdet{o'_k}{(i)}$ (see \cite[Chapter 7]{Ka}). Then the lemma follows by \cite[Theorem 5]{Ka} and the (obvious) fact that a direct summand of an abelian group is a pure subgroup.  
\end{proof}

For any $m\in \mathbb{N}$ and any $i\geq 1$, set 
$$\Psi(k,p^m,i):=\mathrm{ker}\left(\wket{k}{(i)}/p^m\to \hdet{o'_k}{(i)}/p^m\right).$$ 
This notation is reminiscent of $\Psi(k)=\mathrm{ker}\left((\xk)_\Gamma\to A'_k\right)$ and the next proposition partially explains this relation.  We will need the following observation: for every $i\ne 1$, Schneider's isomorphism 
$\wket{k}{(i)}\cong \left(\xK(j)\right)_{\mathcal{G}}$ can also be obtained taking projective limits over $m>0$ of the isomorphisms
\begin{equation}\label{modpn}
\wketp{k}{m}{(i)}\cong \left(A_{k(\mu_{p^m}^{\otimes j})}'/p^m(j)\right)_{\mathrm{Gal}(k(\mu_{p^m}^{\otimes j})/k)}.
\end{equation} 
Using the definition of $k(\mu_{p^m}^{\otimes j})$, the isomorphism (\ref{modpn}) follows from class field theory and Poitou-Tate duality (see \cite[Section 3]{Ca}).

\begin{prop}
For every integer $m$ and $j+1=i\geq 2$, there is an exact sequence
$$0\to \Psi(k,p^m,i)\to \left(\xK/p^m(j)\right)_{\mathcal{G}}\to \left(A'_{k(\mu_{p^m}^{\otimes j})}/p^m(j)\right)_{\mathrm{Gal}(k(\mu_{p^m}^{\otimes j})/k)}\to D_{m,i}\to 0$$
where $D_{m,i}$ is the Pontryagin dual of the kernel of $\hzet{k}{(-j)}/p^m\to\underset{v\mid p}{\oplus}\hzet{k_v}{(-j)}/p^m$ and is trivial if $i\not\equiv1 \pmod{[K:k]}$. 
\end{prop}
\begin{proof} 
Taking cohomology of the exact sequence 
$$0\to \zp(i)\stackrel{p^m}{\longrightarrow}\zp(i)\to\mathbb{Z}/p^m\mathbb{Z}(i)\to 0$$
and using the long exact sequence of Poitou-Tate, we get the following commutative diagram with exact rows
$$
\begin{CD}
0@>>>\Psi(k,p^m,i)@>>>\wket{k}{(i)}/p^m @>>> \hdet{o'_k}{(i)}/p^m @>>> \tilde\oplus_k/p^m @>>>0\,\,\\
@.@.@VVV@VVV@VVV\\
@.0@>>> \wketp{k}{m}{(i)} @>>> \hdetp{o'_k}{m}{(i)}@>>>\tilde\oplus_{k,p^m}@>>>0, 
\end{CD}
$$
where 
\begin{eqnarray*}
\tilde\oplus_{k}=&\!\!\!\!\!\!\!&\mathrm{ker}\left(\underset{v\mid p}{\oplus}\hdet{k_v}{(i)}\to \hzet{k}{(-j)}^*\right),\\
\tilde\oplus_{k,p^m}=&\!\!\!\!\!\!\!&\mathrm{ker}\left(\underset{v\mid p}{\oplus}\hdetp{k_v}{m}{(i)}\to \hzetp{k}{m}{(-j)}^*\right).
\end{eqnarray*}
The vertical arrow in the middle is an isomorphism, since $cd_p(Spec(o'_k))\leq 2$. Then the snake lemma and the discussion prior to Lemma \ref{splitpure} give the exact sequence 
$$0\to \Psi(k,p^m,i)\to \left(\xK/p^m(j)\right)_{\mathcal{G}}\to \left(A'_{k(\mu_{p^m}^{\otimes j})}/p^m(j)\right)_{\mathrm{Gal}(k(\mu_{p^m}^{\otimes j})/k)}\to D_{m,i}\to 0,$$
where $D_{m,i}=\ker\left(\tilde\oplus_k/p^m\to \tilde\oplus_{k,p^m}\right)$. We also have a commutative diagram with exact rows and columns
$$
\begin{CD}
@.0\\
@.@VVV\\
@.D_{m,i}\\
@.@VVV\\
@.\tilde \oplus_k/p^m @>>> \underset{v\mid p}{\oplus}\hdet{k_v}{(i)}/ p^m @>>> \hzet{k}{(-j)}^*/p^m @>>>0\,\,\\
@.@VVV@VVV@VVV\\
0@>>> \tilde\oplus_{k,p^m} @>>> \underset{v\mid p}{\oplus}\hdetp{k_v}{m}{(i)} @>>>\hzetp{k}{m}{(-j)}^*@>>>0 .
\end{CD}
$$
Since $cd_p(Spec(k_v))\leq 2$ for any $v\mid p$, we deduce that the vertical arrow in the middle is an isomorphism. Hence, using the snake lemma and local duality,
\begin{eqnarray*}
D_{m,i}&=&\ker\left(\tilde\oplus_k/p^m \to \underset{v\mid p}{\oplus}\hdet{k_v}{(i)}/p^m\right)\\
&=&\mathrm{coker}\left( \underset{v\mid p}{\oplus}\hdet{k_v}{(i)}[p^m]\to \hzet{k}{(-j)}^*[p^m]\right)\\
&=&\mathrm{coker}\left( \underset{v\mid p}{\oplus}\hzet{k_v}{(-j)}^*[p^m]\to \hzet{k}{(-j)}^*[p^m]\right)\\
&=&\mathrm{coker}\left( \underset{v\mid p}{\oplus}\left(\hzet{k_v}{(-j)}/p^m\right)^*\to \left(\hzet{k}{(-j)}/p^m\right)^*\right)\\
&=&\left(\ker\left(\hzet{k}{(-j)}/p^m\to\underset{v\mid p}{\oplus}\hzet{k_v}{(-j)}/p^m \right)\right)^*,
\end{eqnarray*}
where for an abelian group $M$, we denote by $M[p^m]$ the elements of $M$ annihilated by $p^m$. Note that, if $i\not\equiv1 \pmod{[K:k]}$, then $\hzet{k}{(-j)}=0$ and hence $D_{m,i}$ is trivial for every $m\in \mathbb{N}$. 
\end{proof}

The above proposition and Lemma \ref{splitpure} give the main result of this section.
 
\begin{teo}\label{criterion}
Let $k$ be a number field. For any integer $m$ and $j+1=i\geq 2$, the inclusion $\wket{k}{(i)}\to \hdet{o'_k}{(i)}$ splits if and only if
\begin{equation}\label{comeinca}
\mathrm{ker}\left(\left(\xK/p^m(j)\right)_{\mathcal{G}}\to \left(A'_{k(\mu_{p^m}^{\otimes j})}/p^m(j)\right)_{\mathrm{Gal}(k(\mu_{p^m}^{\otimes j})/k)}\right)=0,\quad \textrm{for every $m\in\mathbb{N}$.}
\end{equation}
\end{teo}

\begin{remark}
It is interesting to compare (\ref{comeinca}) with the condition
$$\left(A'_{k(\mu_{p^m}^{\otimes j})}/p^m(j)\right)_{\mathrm{Gal}(k(\mu_{p^m}^{\otimes j})/k)}=0,\quad \textrm{for every $m\in\mathbb{N}$,}
$$
which is equivalent to the splitting of the inclusion $\hdet{o'_k}{(i)}\to H^2_{cont}(k,\zp(i))$ by \cite[Theorem 1]{Ca}.
\end{remark}

Using Theorem \ref{criterion}, we can give a different proof of Corollary \ref{simplecase}.
 
\begin{proof}[Alternate proof of Corollary \ref{simplecase}.]
Fix $m\in \mathbb{N}$ and write $[k(\mu_{p^m}^{\otimes j}):k]=d_jp^n$, with $(d_j,p)=1$. Then $n$ is the smallest integer such that $k(\mu_{p^m}^{\otimes j})\subseteq K_n$ (in particular $\mathrm{Gal}(K_\infty/k(\mu_{p^m}^{\otimes i}))$ acts trivially on $\mathbb{Z}/p^m\mathbb{Z}(j)$). 
Since $n_0(K)=0$, we have an exact sequence 
$$0\to \Psi(K_n)\to (\xK)_{\Gamma_n}\to A'_{K_n}\to 0.$$
Taking $(-j)$-components with respect to the action of $\Delta$, we get an exact sequence
$$0\to \Psi(K_n)^{[-j]}\to (\xKj{-j})_{\Gamma_n}\to \aKnj{-j}\to 0$$
(note that the actions of $\Gamma_n$ and $\Delta$ commute since $\mathcal{G}_n$ is abelian). Since $\Psi(K)(j)^{\Delta}=0$, we deduce that $\Psi(K_n)(j)^{\Delta}=\Psi(K_n)^{[-j]}=0$ by Lemma \ref{CPsi}. Therefore the above exact sequence reduces to an isomorphism 
$$(\xKj{-j})_{\Gamma_n}\cong \aKnj{-j}.$$
Now applying the functor $\otimes \mathbb{Z}/p^m\mathbb{Z}(j)$ we get an isomorphism
$$\left(\xK/p^m(j)\right)_{\mathcal{G}_n}\to \left(A'_{K_n}/p^m(j)\right)_{\Delta},$$
since $\Gamma_n$ acts trivially on $\mathbb{Z}/p^m\mathbb{Z}(j)$. 
Finally, taking $\mathrm{Gal}(K_n/K)$-coinvariants we get an isomorphism
$$\left(\xK/p^m(j)\right)_{\mathcal{G}}\cong \left(A'_{K_n}/p^m(j)\right)_{\mathrm{Gal}(K_n/k)}=\left(A'_{k(\mu_{p^m}^{\otimes j})}/p^m(j)\right)_{\mathrm{Gal}(k(\mu_{p^m}^{\otimes j})/k)},$$
where the last equality comes from the fact that $\mathrm{Gal}(K_n/k(\mu_{p^m}^{\otimes j}))$ acts trivially on $\mathbb{Z}/p^m\mathbb{Z}(j)$ and has order coprime with $p$.
\end{proof}

\section{Examples}\label{examples}
In this section we will illustrate our results in the case of quadratic number fields and $p=3$. We further assume $i=2$ to recover the classical context. In other words (see the introduction for more details), we shall consider the exact sequence 
\begin{equation}\label{sekt}
0\to W\!K_2(k)\{3\}\to K_2(o_{k})\{3\}\to \underset{v\mid 3}{\tilde\oplus}\mu(k_v)\{3\}\to 0,
\end{equation}
where 
$$\underset{v\mid 3}{\tilde\oplus}\mu(k_v)\{3\}=\mathrm{ker}\left(\underset{v\mid 3}{\oplus}\mu(k_v)\{3\}\to \mu(k)\{3\}\right).$$
Using PARI \cite{PARI}, we will provide examples for both split and non-split cases of the above exact sequence. 

Let $k=\mathbb{Q}(\sqrt{\delta})$ be a quadratic field. As usual we set $K=k(\mu_3)$ and let $k'=\mathbb{Q}(\sqrt{-3\delta})$ denote the quadratic subfield of $K=k(\mu_3)$ which is different from $k$ and $\mathbb{Q}(\mu_3)$. Note that in this situation $n_0(K)=0$, \textit{i.e.} all the primes above $3$ in $K_\infty/K$ are totally ramified. Moreover, we have an isomorphism 
$$\Psi(K)(1)^\Delta\cong\Psi(k').$$ 

\begin{remark}\label{comune}
For a quadratic number field $k$, Schneider's isomorphism $W\!K_2(k)\{3\}\cong\xK\!\!(1)_\mathcal{G}=\xKj{-1}(1)_\Gamma,$ together with Nakayama's lemma, implies that $W\!K_2(k)\{3\}=0$ if and only if $X'_{k'_\infty}=0$, since $\xKj{-1}\cong X'_{k'_\infty}$. Moreover, if $\Psi(k')=0$, then there is an isomorphism $(X'_{k'_\infty})_\Gamma\cong A'_{k'}$. In particular, when $\Psi(k')=0$, $W\!K_2(k)\{3\}=0$ if and only if $A'_{k'}=0$. 
\end{remark}

It is easy to see that, if $k=\mathbb{Q}(\sqrt{\delta})$ with $\delta\in \mathbb{Z}$ square-free, then $\tilde\oplus_{v\mid 3}\mu(k_v)\{3\}$ is non trivial if and only if $\delta\equiv -3 \,\,(3^2)$ and $\delta\ne -3$. In fact in these circumstances there is only one prime $v$ above $3$ and $\mu(k_v)\{3\}$ has order $3$ while $\mu(k)\{3\}$ is trivial, so that $[K_2(o_{k})\{3\}:W\!K_2(k)\{3\}]=3$. 

Thus for the rest of this section we set $k=\mathbb{Q}(\sqrt{\delta})$, with $\delta\in D$ where
$$D=\{x\in\mathbb{Z}\setminus\{-3\}\,|\,\textrm{$x$ square-free, $x\equiv -3 \!\!\pmod{3^2}$}\}.$$ 

\begin{example}\label{splitcase}
Suppose that
\begin{equation}\label{reqsplit}
\Psi(k')=\mathrm{ker}((X'_{k'_\infty})_\Gamma\to A'_{k'})=0\quad \textrm{and}\quad A'_{k'}\ne 0.
\end{equation}
The first of the above conditions can be easily verified using the methods of \cite[Section 4.3]{LMN}. Then (\ref{sekt}) splits with $W\!K_2(k)\{3\}\ne0$, thanks to Corollary \ref{simplecase} and Remark \ref{comune}, and, since $\delta\in D$, we have 
$$K_2(o_k)\{3\}\cong W\!K_2(k)\{3\}\oplus\mathbb{Z}/3\mathbb{Z}$$
as abelian groups (an easy argument shows that the same holds with $k$ replaced by $K$). 

For positive $\delta$, combined with the Birch-Tate formula $\zeta_k(-1)=(\#K_2(o_k))/24$, the splitting of (\ref{sekt}) sometimes allows to completely determine the structures of $K_2(o_k)\{3\}$ and $W\!K_2(k)\{3\}$. For example, if $k=\mathbb{Q}(\sqrt{3\cdot 239})$, the Birch-Tate formula gives $\#K_2(o_k)\{3\}=3^2$. We also have $A'_{k'}\cong \mathbb{Z}/3\mathbb{Z}$, which implies that the $3$-rank of $K_2(o_k)$ is $2$, thanks to the Keune-Tate exact sequence (\cite[proof of Theorem 6.6]{Ke} or \cite[Theorem 6.2]{Ta2}). Finally $\Psi(k')=0$ and therefore we get
$$W\!K_2(k)\{3\}\cong\mathbb{Z}/3\mathbb{Z}\quad\textrm{and}\quad K_2(o_{k})\{3\}\cong \mathbb{Z}/3\mathbb{Z}\oplus \mathbb{Z}/3\mathbb{Z}.$$  

Based on the computation of the regulator $R_2(k)$ in Lichtenbaum's generalization of the Birch-Tate formula $\zeta_k(-1)=R_2(k)(\#K_2(o_k))/24$, Browkin and Gangl (\cite{BG}) predicted the structure of $K_2(o_k)$ when $k$ is an imaginary quadratic field whose discriminant has absolute value less than $5000$. We have applied our methods to the cases in the list of \cite{BG} and, as far as the $3$-parts are concerned, our results agree with those in their list. 
 
\end{example}

\begin{example}
If $\Psi(k')\ne 0$, we cannot apply Corollary \ref{simplecase}. Nevertheless, if  \begin{equation}\label{reqnsplit}
\Psi(k')=\mathrm{ker}((X'_{k'_\infty})_\Gamma\to A'_{k'})\ne0\quad \textrm{and}\quad A'_{k'}= 0,
\end{equation}
then (\ref{sekt}) does not split. Indeed, when $A'_{k'}=0$, the condition $\Psi(k')\ne 0$ is equivalent to $X'_{k'_\infty}\ne 0$. Therefore
\begin{equation}\label{wknt}
\mathrm{ker}\left(\left(\xK/3(1)\right)_{\mathcal{G}}\to \left(A'_{K}/3(1)\right)_{\Delta}\right)=\left(\xK/3(1)\right)_{\mathcal{G}}\cong X'_{k'_\infty}/3(1)\ne 0,
\end{equation}
which implies that (\ref{sekt}) does not split by Theorem \ref{criterion}. In fact, this can also be proved in the following way: if $A'_{k'}=0$, then the $3$-rank of $K_2(o_k)$ is $1$ (by the Keune-Tate exact sequence) and, since $W\!K_2(k)\{3\}\ne 0$ by (\ref{wknt}), this implies that (\ref{sekt}) does not split. Note that the first condition in (\ref{reqnsplit}) can be verified using the recipe of \cite[Section 4.3]{LMN}. If (\ref{reqnsplit}) holds, then, since $\delta\in D$, we have 
$$W\!K_2(k)\{3\}\cong\mathbb{Z}/3^a\mathbb{Z}\quad\textrm{and}\quad K_2(o_{k})\{3\}\cong \mathbb{Z}/3^{a+1}\mathbb{Z}$$
for some $a\in\mathbb{N}$.

As in the previous example in the totally real case we can use the information about the non splitting of (\ref{sekt}) to completely determine the structures of $K_2(o_k)\{3\}$ and $W\!K_2(k)\{3\}$. For example, when $k=\mathbb{Q}(\sqrt{3\cdot14})$,  the Birch-Tate formula gives $\#K_2(o_k)\{3\}=3^3$. We also have $A'_{k'}=0$, which implies that $K_2(o_k)\{3\}$ is cyclic. Finally $\Psi(k')\cong\mathbb{Z}/3\mathbb{Z}\ne 0$ and therefore we get
$$W\!K_2(k)\{3\}\cong\mathbb{Z}/3^2\mathbb{Z}\quad\textrm{and}\quad K_2(o_{k})\{3\}\cong \mathbb{Z}/3^3\mathbb{Z}.$$ 

In the imaginary case, the results given by our methods agree once more with the predictions of \cite{BG}, as far as the $3$-parts are concerned.

\end{example}

\begin{remark}
If $\Psi(k')\ne0$ and $A'_{k'}\ne0$, then the inclusion $W\!K_2(k)\{3\}\subseteq K_2(o_{k})\{3\}$  may or may not split, indeed both cases may occur. For instance, for $\delta=3\cdot 1409$ or $3\cdot 1658$, we have $\#\Psi(k')=\#A'_{k'}=3$. Moreover the Birch-Tate formula gives $\#K_2(o_{k})\{3\}=3^3$ and the Keune-Tate exact sequence implies that $K_2(o_{k})\{3\}$ has $3$-rank $2$ and its exponent must therefore be $3^2$.
In particular (\cite[Theorem 6.6]{Ke}), we have 
\begin{equation}\label{last}
W\!K_2(k)\{3\}\cong \left(A'_{K_1}\otimes\mu_9\right)_{\mathrm{Gal}(K_1/k)}.
\end{equation}
Now PARI tells us that the right-hand term of (\ref{last}) is isomorphic to $\mathbb{Z}/3\mathbb{Z}\times\mathbb{Z}/3\mathbb{Z}$ (resp. $\mathbb{Z}/3^2\mathbb{Z}$) when $\delta=3\cdot 1409$ (resp. $\delta=3\cdot 1658$). This implies that the inclusion $W\!K_2(k)\{3\}\subseteq K_2(o_{k})\{3\}$ does not split (resp. splits) when $\delta=3\cdot 1409$ (resp. $\delta=3\cdot 1658$).  
\end{remark}

\bigskip
\noindent
\textbf{Acknowledgments}

We are very grateful to Thong Nguyen Quang Do for several helpful exchanges during the preparation of this article. We would also like to thank Herbert Gangl for providing us with the program used to compute the table of \cite{BG}, along with useful suggestions on how to use it. Finally, we thank Filippo Nuccio for his remarks on a preliminary version.


\begin{tabularx}{\textwidth}{XX}
   Luca Caputo & Abbas Movahhedi\\
   Dipartimento di Matematica & Facult\'e de Sciences et Techniques -XLIM\\
   via Filippo Buonarroti 1/c & 123 Avenue Albert Thomas\\
   56127 Pisa & 87060 Limoges\\
   Italy & France\\
   caputo@mail.dm.unipi.it & mova@unilim.fr
\end{tabularx}


\begin{thebibliography}{999999}
\begin{footnotesize}
\linespread{0.7}
\bibitem[Ba93]{Ba}\textsc{G. Banaszak}, {\it Generalization of the Moore exact sequence and the wild kernel for higher K-groups}, Compositio Math. 86 no. 3 (1993), 281-305.
\bibitem[BG99]{BG}\textsc{J. Browkin and H. Gangl}, {\it Tame and wild kernels of quadratic imaginary number fields},  Math. Comp. 68 n. 225  (1999), 291-305.
\bibitem[Ca10]{Ca}\textsc{Caputo L.}, \textit{Splitting in the $K$-theory localization sequence of number fields}, J. Pure Appl. Algebra 215 n. 4 (2010), 485-495.
\bibitem[Co72]{Co}\textsc{J. Coates}, {\it On $K_{2}$ and some classical conjectures in algebraic number theory},  Ann. of Math. (2)  95  (1972), 99-116.
\bibitem[GJ85]{GJ}\textsc{M. Grandet et J.-F. Jaulent}, {\it Sur la capitulation dans une $\mathbb{Z}_\ell$-extension}, J. Reine Angew. Math. 362 (1985), 213-217. 
\bibitem[Gr78]{Gr}\textsc{R. Greenberg}, {\it A note on $K_2$ and the theory of $\zp$-extensions},  Amer. J. Math. 100  n. 6  (1978), 1235-1245.
\bibitem[Hu01]{Hu}\textsc{K. Hutchinson}, \emph{The $2$-Sylow subgroup of the wild kernel of exceptional number fields}, J. Number Theory 87 (2001), 222-238.
\bibitem[Iw73]{Iw}\textsc{K. Iwasawa}, {\it On $\mathbb{Z}_l$-extensions of algebraic number fields},  Ann. of Math. (2)  98  (1973), 246-326.
\bibitem[Ja89]{Ja}\textsc{U. Jannsen}, {\it Iwasawa modules up to isomorphisms}, Algebraic number theory, 171-207, Adv. Stud. Pure Math., 17, Academic Press, Boston, MA, 1989. 
\bibitem[Ka69]{Ka}\textsc{I. Kaplansky}, {\it Infinite abelian groups}, The University of Michigan Press, Ann Arbor, MI, 1969. 
\bibitem[Ke89]{Ke}\textsc{F. Keune}, \emph{On the structure of the $K_2$ of the ring of integers in a number field}, $K$-Theory 2 no. 5 (1989), 625-645.
\bibitem[Kol93]{Kol1}\textsc{M. Kolster}, {\it  Remarks on \'etale $K$-theory and Leopoldt's conjecture}, S\'{e}minaire de Th\'{e}orie des Nombres, Paris 1991-1992 (S. David, ed.),  37-62, Progr. Math., vol. 116, Birkh\"{a}user, Boston, Boston, MA, 1993.
\bibitem[Kol04]{Kol2}\textsc{M. Kolster}, \emph{$K$-theory and arithmetic}, in \emph{Contemporary developments in algebraic K-theory}, 191-258 (electronic), ICTP Lect. Notes, XV, Abdus Salam Int. Cent. Theoret. Phys., Trieste, 2004.
\bibitem[KM00]{KM}\textsc{M. Kolster and A. Movahhedi}, {\it Galois Co-descent for Etale Wild Kernels and Capitulation}, Ann. Inst. Fourier (Grenoble) 50 n. 1 (2000), 35-65.
\bibitem[KM13]{KM2}\textsc{M. Kolster and A. Movahhedi}, \textit{On $\lambda$-invariants of number fields and \'etale cohomology}, to appear in J. $K$-Theory.
\bibitem[LMN05]{LMN}\textsc{M. Le Floc'h, A. Movahhedi, T. Nguyen Quang Do}, {\it On capitulation cokernels in Iwasawa theory}, Amer. J. Math. 127  n. 4 (2005), 851-877. 
\bibitem[NQD86]{N1}\textsc{T. Nguyen Quang Do}, {\it Sur la $\zp$-torsion de certains modules galoisiens}, Ann. Inst. Fourier (Grenoble) 36 n. 2 (1986), 27-46.
\bibitem[NQD88]{N2}\textsc{T. Nguyen Quang Do}, {\it Sur la torsion de certains modules galoisiens II}, S\'eminaire de Th\'eorie des Nombres, Paris 1986-87,  271-297, Progr. Math., 75, Birkh\"auser Boston, Boston, MA, 1988.
\bibitem[NQD89]{N3}\textsc{T. Nguyen Quang Do}, {\it Sur la cohomologie de certains modules galoisiens $p$-ramifi\'es}, Th\'eorie des Nombres, J.-M. De Koninck and C. Levesque (ed.), C.R. Conf. Int., Qu\'ebec/Can, 1987, 740-754, 1989.
\bibitem[NQD92]{N4}\textsc{T. Nguyen Quang Do}, {\it Analogues sup\'erieurs du noyau sauvage}, J. Th\'eor. Nombres Bordeaux 4 no. 2 (1992), 263-271.
\bibitem[PA]{PARI}\textsc{The PARI~Group}, {\it {PARI/GP, version {\tt 2.3.4}}}, Bordeaux, 2008,
\newblock available at \url{http://pari.math.u-bordeaux.fr/}.
\bibitem[Sc79]{Sc}\textsc{P. Schneider}, {\it \"Uber gewisse Galoiscohomologiegruppen}, Math. Z. 168 n. 2 (1979), 181-205.
\bibitem[So79]{So}\textsc{C. Soul\'e}, \emph{$\mathrm{K}$-th\'eorie des anneaux d'entiers de corps de nombres et cohomologie \'etale}, Invent. Math. 55 n. 3 (1979), 251-295.
\bibitem[Ta71]{Ta1}\textsc{J. Tate}, \emph{Symbols in Arithmetic} in \emph{Actes, Actes du Congr\`es International des Math\'ematiciens (Nice, 1970), Tome 1}, Gauthier-Villars, 1971, 201-211.
\bibitem[Ta76]{Ta2}\textsc{J. Tate}, \emph{Relation between $K_2$ and Galois cohomology}, Invent. Math. 36 (1976), 257-274. 
\bibitem[Va08]{Va1}\textsc{R. Validire}, {\it Capitulation des noyaux sauvages \'etales}, Th\`ese de l'Universit\'e de Limoges, 2008.
\bibitem[Va09]{Va2}\textsc{R. Validire}, {\it Capitulation for even $K$-groups in the cyclotomic $\zp$-extension}, J. Th\'eor. Nombres Bordeaux 21 n. 2 (2009), 441-456.
\bibitem[Wa97]{Wa}\textsc{L. C. Washington}, {\it Introduction to cyclotomic fields}, GTM 83, Springer-Verlag, 1997.
\bibitem[We09]{We}\textsc{C. Weibel}, \textit{The norm residue isomorphism theorem}, J. Topol. 2 (2009), no. 2, 346-372. 
\end{footnotesize}
\end{thebibliography}
\end{document}